\documentclass[oneside,12pt]{amsart}

\usepackage{amsmath, amssymb}
\usepackage[utf8]{inputenc}
\usepackage{xcolor}
\usepackage{tikz}
\usepackage{tikz-cd}

\definecolor{dblue}{rgb}{0.0, 0.0, 0.55}

\usepackage[colorlinks, urlcolor=black]{hyperref}
\hypersetup{citecolor=dblue, linkcolor=dblue}

\usepackage{datetime2}
\newtheorem{theorem}{Theorem}[section]

\newtheorem{proposition}[theorem]{Proposition}

\newtheorem{remark}[theorem]{Remark}

\numberwithin{equation}{section}

\title[On generic representations over local fields of positive characteristic]{On generic representations of quasi-split reductive groups over local fields of positive characteristic}


\author{H\'ector del Castillo}

\author{Guy Henniart}

\author{Luis Lomel\'i}
\date{\today}

\begin{document}

\begin{abstract}
Let $F$ be a locally compact non-Archimedean field, and $\bf G$ a connected quasi-split reductive group over $F$. We are interested in complex irreducible smooth generic representations $\pi$ of ${\bf G}(F)$. When $F$ has positive characteristic, we prove important properties which previously were only available for $F$ of characteristic 0. The first one is the tempered $L$-function conjecture of Shahidi, stating that when $\pi$ as above is tempered, then the $L$-functions attached to $\pi$ by the Langlands-Shahidi method have no pole for ${\rm Re}(s)>0$. We also establish the standard module conjecture of Casselman and Shahidi, saying that if $\pi$ is written as the Langlands quotient of a standard module, then it is in fact the full standard module. Finally, for a split classical group $\bf G$ we prove a useful result on the unramified unitary spectrum of ${\bf G}(F)$.
\end{abstract}

\maketitle

\section*{Introduction}

\parindent=17pt
\parskip=7pt

Let $F$ be a locally compact non-Archimedean field, and $\bf G$ a connected quasi-split reductive group over $F$. We are interested in complex irreducible smooth generic representations $\pi$ of ${\bf G}(F)$, in view of their applications to automorphic forms. The Langlands-Shahidi (in short, LS) method constructs $L$-functions and local factors of a complex variable associated to $\pi$. It was developed by Shahidi in the 1980’s and further \cite{Sha1990}, but was restricted to the case where $F$ has characteristic 0. Now, the third-named author has extended the LS method to positive characteristic in \cite{LomLS}. However, several basic facts, put forward by Shahidi and established subsequently for $F$ of characteristic 0, are not in general available in the literature when $F$ has positive characteristic, and it is our purpose to prove them here. We mention the special case when $\bf G$ is a non-split quasi-split special orthogonal group in $2n$ variables, where the results presented in this article are used in work of the first-named author \cite{dC}, who establishes the transfer to ${\rm GL}_{2n}$ of generic cuspidal automorphic representations of $\bf G$.

In this paper we establish three main results when $F$ has positive characteristic. The first one is the so-called \emph{tempered $L$-function conjecture}, stating that if the representation is tempered, then the $L$-functions produced by the LS-method are holomorphic on ${\rm Re}(s)>0$. The second one is \emph{the standard module conjecture}, saying that if one writes $\pi$ as the Langlands quotient of a standard module, then the full standard module is irreducible, hence it is $\pi$ itself. In fact, we deduce the latter result from the former following Heiermann and Opdam; see Section \ref{s:temp}. The third result states that if $\bf G$ is a split classical group and $\bf H$ is the split classical group where ${\bf M} = {\rm GL}_m \times {\bf G}$ is a maximal Levi subgroup, then for a unitary unramified $\pi$ of ${\bf G}(F)$ and an unramified tempered irreducible smooth representation $\tau$ of ${\rm GL}_m(F)$, the unramified irreducible component of the representation of ${\bf H}(F)$ parabolically induced from $\sigma = \tau \nu^s \otimes \pi$ of ${\bf M}(F)$, where $\nu = \left| \det(\cdot) \right|_F$, is not unitary if ${\rm Re}(s)>1$.

For the tempered $L$-function conjecture and the standard module conjecture, there are at least two natural methods of proof. The first one is the method of close local fields \`a la Kazhdan. It consists in transferring the results from characteristic zero to positive characteristic by using an isomorphism of the Hecke algebra in ${\bf G}(F)$ of a suitable congruence sugbroup of the Iwahori subgroup, onto the exactly analogous Hecke algebra in ${\bf G}'(F')$, where $F'$ is a locally compact field of characteristic 0 close enough to $F$, and ${\bf G}'$ is a kind of lift of $\bf G$ to $F'$. At present, suitable isomorphisms have been obtained by Ganapathy only when $\bf G$ is split \cite{Ga2015}, in which case $\bf G$ and ${\bf G}'$ come from the same Chevalley group, but see  partial results in \cite{Ga2019,Ga2021} towards the general case. When $\bf G$ is split, our first two theorems in characteristic $p$ were obtained with this approach in \cite{Lom2019}, using \cite{Ga2015}. Note that Ganapathy and Varma in \cite{GaVa2017} apply the Kazhdan method to the very important question of transporting the Langlands correspondence for split classical groups to the positive characteristic case.

It is the method of close local fields that we use to obtain our third result. However, we first work in characteristic zero and prove the non-unitarity criterion, since it is new for any non-archimedean local field. We rely on the classification of the unramified unitary dual of split classical groups, due to Mui\'c and Tadi\'c over a non-archimedean local field of characteristic $0$ \cite{MuTa2011}. After proving our theorem in characteristic 0, we then transfer \`a la Kazhdan to include local fields of characteristic $p$.

The second method of proof, which is the one we use for our first two results, consists in adapting to positive characteristic the proofs in characteristic zero. For the tempered $L$-function conjecture in characteristic 0, the method of Heiermann and Opdam \cite{HeOp2013}, exploits results of Silberger on the Plancherel measure, and of Heiermann on the location of discrete series, to translate the problem to a Hecke algebra defined by generators and relations, and then the result for that algebra comes from the fact that it is a Hecke algebra occurring in groups for which the conjecture is already known. Now the results of Silberger and Heiermann are valid in all characteristics, and the translation to an explicit Hecke algebra only necessitates basic facts on $L$-functions for cuspidal representations, due to Shahidi in characteristic 0 \cite{Sha1990}, and extended to positive characteristic by the third-named author \cite{LomLS}. In characteristic zero, Heiermann and Mui\'c reduced the standard module conjecture to the tempered $L$-function conjecture in \cite{HeMu2007}, and our results on $L$-functions allow the same derivation in positive characteristic.

The layout of the paper is as follows. In Section~\ref{s:results} we precisely state the results mentioned above. Section~\ref{s:setup} gathers basic notions and properties. It is in Section~\ref{s:temp} where we prove the first two results, after recalling how to define LS $L$-functions for possibly non-tempered generic representation. This involves an analytic continuation  from the tempered generic case to the  general generic case. The coherence of that process uses the tempered LS $L$-fuction conjecture. Section~\ref{s:L-function} explains how to extract LS $L$-functions from the LS $\gamma$-factor. In Section~\ref{s:unr}, we show that for unramified generic irreducible generic representations, the LS factors are the same as those of the corresponding Langlands parameters. This involves proving a folklore result (Proposition \ref{prop:unrtemp}) on unramified representation as Langlands quotients. The proof of that was explained to us by J.L. Waldspurger. In Section~\ref{s:nr:non-uni} we prove the third main result when $F$ has characteristic $0$ and in Section~\ref{s:Kazhdan} we transfer that result to positive characteristic via the method of close local fields. Finally, in the Appendix we explore the structure of unramified representations of $\operatorname{SL}_2$ and $\operatorname{SU}_{2,1}$.

\subsection*{Notation} In this paper, $F$ is a non-Archimedean locally compact field. We use the customary notation $\mathcal{O}_F$ for its ring of integers and $\mathfrak{p}_F$ for the maximal ideal of $\mathcal{O}_F$. We let $p$ be the characteristic of the residue field $\kappa_F = \mathcal{O}_F/\mathfrak{p}_F$, and let
$q_F = {\rm card}(\kappa_F)$. We write $\left| \cdot \right|_F$ for the normalized absolute value of $F$, where we sometimes use exponential notation \[ \left| x \right|_F = q_F^{-{\rm ord}_F(x)}, \]
with ${\rm ord}_F : F \rightarrow \mathbb{Z}$. For any $n
\in \mathbb N$,  we let $\nu_n=\nu$ be the character of ${\rm GL}_n(F)$ given by $\nu(g) = \left| \det(g) \right|_F$. We also choose a non-trivial character $\psi$ of $F$. For a complex variable $s$, we often write $t = q^{-s}$. Given an algebraic group $\bf H$ over $F$, we usually write $H={\bf H}(F)$, and if $\bf H$ is connected reductive we let ${\bf H}_{\rm ad}$ denote its adjoint group.

Let $\bf G$ be a connected reductive group over $F$. Most often we assume $\bf G$ quasi-split. We choose a maximal split torus $\bf S$ of $\bf G$, and write $\bf Z$ for its centralizer. When $\bf G$ is quasi-split $\bf Z$ is a torus that we write $\bf T$. We choose a minimal parabolic subgroup $\bf B$ of $\bf G$ with Levi subgroup $\bf Z$, and write $\bf U$ for its unipotent radical. We write $W$ for the (little) Weyl group $N_{G}({S})/{Z}$. We write $\Phi$ for the set of roots of $\bf S$ in $\bf G$, $\Delta$ for the subset of simple roots with respect to $\bf B$; we denote the set of positive roots by $\Phi^+$. For a simple root $\alpha$, we let $\alpha^\vee$ be the corresponding coroot. For a parabolic subgroup $\bf P$ of $\bf G$ containing $\bf Z$, we understand ${\bf P}={\bf MN}$ to mean that $\bf N$ is the unipotent radical of $\bf P$, and $\bf M$ the Levi subgroup of $\bf P$ containing $\bf Z$.

We also choose a separable algebraic closure $\overline{F}$ of $F$, and write $\Gamma_F$ for the Galois  group of $\overline{F}$ over $F$ and $\mathcal{W}_F$ for the Weil group of $\overline{F}$ over $F$. The $L$-group of $\bf G$ is a semidirect product
\begin{equation*}
   {}^LG = {}^LG^\circ \rtimes \mathcal{W}_F
\end{equation*}
where ${}^LG^\circ$ is the dual group of $\bf G$ (in \cite{BCor1979}  Borel uses $\Gamma_F$, but we prefer using  $\mathcal{W}_F$, as it is more customary nowadays). 
When $\bf P=MN$ is a parabolic subgroup of $\bf G$  containing $\bf B$, the $L$-group of $\bf P$ is then of the form ${}^LP = {}^LP^\circ \rtimes \mathcal{W}_F$, which is an extension of the Levi ${}^LM = {}^LM^\circ \rtimes \mathcal{W}_F$ by a unipotent radical ${}^LN = {}^LN^\circ$. We write ${}^L\mathfrak{n}$ for the Lie algebra of the unipotent radical ${}^LN$.

By a representation of a locally profinite group like $G$, we always mean a smooth complex representation of $G$. Let $P=MN$ be a parabolic subgroup of $G$ containing $B$. If $\sigma$ is a representation of $M$ and $\chi$ a character of $M$, extended to $P$ by acting trivially on $N$, then
\[ {\rm I}(\chi,\sigma) = {\rm Ind}_P^G(\chi\sigma) \] denotes the representation obtained via normalized  parabolic induction.

\subsection*{Acknowledgments} We would like to thank J.-L. Waldspurger for mathematical communications. The first author was supported by ANID Postdoctoral Project 3220656. The third author was supported in part by FONDECYT Grant 1212013.

\section{The results}\label{s:results}

 Assume $\bf G$ is quasi-split. We are interested in generic irreducible representations of $G$. For definiteness, we briefly recall the notion of genericity. 
We choose a pinning for the unipotent subgroups corresponding to the simple (absolute) roots of $\bf T$ in $\bf B$ over $\overline{F}$, and assume further that the pinning is equivariant with respect to the action of $\Gamma_F$. From $\psi$ and the pinning we obtain a character $\theta_\psi$ of $U$, which is trivial on the normal subgroup $\prod_{\alpha \in \Phi^+ \setminus \Delta} U_\alpha$. An irreducible smooth representation $\pi$ of $G$ is called $\theta_\psi$-generic if there is a non-trivial $U$-morphism from $\pi$ to $\theta_\psi$ (then those morphisms form a line). We say that $\pi$ is \emph{generic} if it is $\theta_\psi$-generic for some choice of pinning as above; varying $\psi$ amounts to changing the pinning.

The \emph{tempered $L$-function conjecture} (Theorem~\ref{TLC} below) was proved in \cite{HeOp2013} when $F$ has characteristic 0. It uses Langlands-Shahidi $L$-functions, which we now succinctly recall. Let ${\bf P}={\bf MN}$ be a maximal parabolic subgroup of ${\bf G}$ containing ${\bf B}$. Then ${}^LM$ acts by conjugation on ${}^L\mathfrak{n}$; an irreducible component $r$ of this action will be called a Langlands-Shahidi (LS) representation of ${}^LM$. For such an LS representation $r$ and an irreducible generic representation $\sigma$ of $M$, the LS method produces $\gamma$-factors $\gamma(s,\sigma,r,\psi)$ and $L$-functions $L(s,\sigma,r)$, which are rational functions in $t=q_F^{-s}$, for a complex variable $s$, see Section \ref{LS:local:factors}.

\begin{theorem}\label{TLC}
Let $\sigma$ be a tempered generic irreducible representation of $M$ and let $r$ be an LS representation of ${}^LM$. Then $L(s, \sigma, r)$ is holomorphic for ${\rm Re}(s)>0$.
\end{theorem}

The \emph{standard module conjecture} (Theorem~\ref{SMC} below), proved when $F$ has characteristic zero in \cite{HeMu2007}, using \cite{HeOp2013}, does not require $L$-functions in its formulation. It uses the Langlands classification, expressing an irreducible smooth representation $\pi$ of $G$ as the unique quotient ${\rm J}(\chi,\sigma)$ of a  representation of the form ${\rm I}(\chi,\sigma)$, where $P=MN$ is a parabolic subgroup of $G$ containing $B$, $\sigma$ an irreducible tempered representation of $M$ and $\chi$ a character of $M$ with positive real values, which is furthermore strictly positive with respect to $N$. 
It is known that $\bf P$, $\chi$, and the isomorphism class of $\sigma$ are determined by $\pi$. The representation ${\rm I}(\chi,\sigma)$ is known as a standard module \cite{CaSh1998}.

\begin{theorem}\label{SMC}
Let $\pi$ be a generic  representation of $G$, written as the Langlands quotient ${\rm J}(\chi,\sigma)$ of ${\rm I}(\chi,\sigma)$ as above. Then ${\rm I}(\chi,\sigma)$ is irreducible, hence
\[ \pi \simeq {\rm I}(\chi,\sigma). \]
\end{theorem}

Our third main result concerns the case where $\bf G$ is a split classical group. Mui\'c has given a convenient classification of unramified irreducible representations of $G$ \cite{Mu2006}, valid if $F$ has characteristic different than $2$. Mui\'c and Tadi\'c used it to get a classification of the unramified unitary dual of $G$ \cite{MuTa2011} when $F$ has characteristic $0$, see also Barbasch \cite{BaMo1996}. We deduce the following.

\begin{theorem}\label{thm:non-uni}
Let $\bf G$ be a split classical group, $m$ a positive integer. Consider ${\rm GL}_m \times {\bf G}$ as a maximal Levi subgroup of a split classical group $\bf H$. Let $\pi$ be a unitary unramified irreducible representation of $G$ and $\tau$ a tempered unramified irreducible representation of ${\rm GL}_m(F)$. Then for ${\rm Re}(s)>1$, the unramified irreducible component of the parabolically induced representation
\[ I(s,\tau \otimes \pi) = \operatorname{Ind}_{P}^H(\nu^s\tau \otimes \pi) \]
of $H$ is not unitary.
\end{theorem}

Our proof of Theorem~\ref{thm:non-uni} consists of two stages. On the one hand, it is Theorem~\ref{thm:zero:non-uni} over non-Archimedean local fields of characteristic zero. On the other hand, in Section \ref{s:Kazhdan} we use a comparison of local fields à la Kazhdan and transfer to the positive characteristic case.

\section{Mise en place}\label{s:setup}

In this section we recall and establish  basic notions and properties used throughout the article, specially for the proof of Theorem \ref{TLC} and Theorem \ref{SMC}.

\subsection{Root spaces} Let $\bf G$ be a connected reductive group over $F$, whose group of $F$-rational points is denoted by $G={\bf G}(F)$. We let \[X^*({\bf G})=\operatorname{Hom}_{F}({\bf G},\mathbb G_m),\] be the group of algebraic characters of $\bf G$ (i.e. algebraic group homomorphisms over $F$ into $\mathbb G_m$). It is a finitely generated free $\mathbb Z$-module. If $\overline F$ is a separable closure of $F$, then the module $X^*({\bf G}_{\overline F})$ carries a discrete linear action of $\Gamma_F=\operatorname{Gal}(\overline{F}/F)$. We note that $X^*({\bf G})$ is the module of fixed elements of $\Gamma_F$ in $X^*({\bf G}_{\overline F})$. Similarly, we let 
\[X_*({\bf G})=\operatorname{Hom}_{F}(\mathbb G_m,{\bf G}),\] 
be the group of algebraic cocharacters of ${\bf G}$. Finally, let $\left\langle \cdot\, , \cdot \right\rangle$ be the perfect pairing between characters and cocharacters. 

Let $\bf T_G$ be the maximal $F$-split torus in the center of $\bf G$. Restriction gives a homomorphism from $X^*({\bf G})$ to $X^*({\bf T}_{\bf G})$, which is injective with finite cokernel.  We let $a_{\bf G}$ be the dual, as a real vector space, of
\[ a_{\bf G}^*=X^*({\bf T_G})\otimes_{\mathbb Z} \mathbb R. \]
Moreover, $a_{\bf G}$ identifies with $X_*({\bf T_G}) \otimes_{\mathbb Z} \mathbb R$, via $\left\langle \cdot\, , \cdot \right\rangle$.  As in \cite{Wa2003} we define a homomorphism  \[H_G \colon G \to a_{\bf G}\] such that $|\chi(g)|_F = q^{-\langle \chi ,H_G(g)\rangle}$ for every $F$-rational character $\chi \in X^*({\bf G})$ and $g \in G$. Its image is a lattice in $a_{\bf G}$. 
If $\pi$ is a smooth representation of $G$ and $\nu \in a_{\bf G}^*\otimes_{\mathbb R}\mathbb C$, we denote by $\pi_\nu$ the smooth representation of $G$ defined by 
\[\pi_\nu (g) = q^{-\langle \nu, H_G(g)\rangle} \pi(g).\]

Let $G^1 = \ker H_G.$
 It is an open normal subgroup of $G$ containing all compact subgroups of $G$ \cite[V.2.3 Proposition]{Re2010}; in fact, it is generated by them \cite[Remarque 5.12.1, Partie II]{HeLe2017}.
The map
\begin{align}\label{eq:rootspace}
a_{\bf G}^*\otimes_{\mathbb R}\mathbb C&\to \operatorname{Hom}_{\mathbb Z}(G/G^1,\mathbb C^\times)\colon
(\nu \mapsto  (g\mapsto q^{-\langle \nu, H_G(g)\rangle})
\end{align}
is onto. That map has a discrete kernel, whereas its restriction to $a_{\bf G}^*$ yields a one-to-one map to the group of characters of $G$ trivial on $G^1$ with positive real values. More precisely $1_\nu$ is trivial if and only if $\nu$ belongs to  $\frac{2\pi }{\ln(q)}\operatorname{Im}(H_G)^*$, where $\operatorname{Im}(H_G)^*$ is the lattice in $a_{\bf G}^*$ dual to $\operatorname{Im}(H_G)$.

\label{subsec:orthogonal}In our notation, $\bf S$ is a maximal $F$-split torus in $\bf G$, $\bf Z$ its centralizer, and $\bf B$ a minimal  parabolic subgroup of $\bf G$, with Levi subgroup $\bf Z$ and unipotent radical $\bf U$. The group $\bf T_G$ is a subgroup of $\bf S$ and $X_*({\bf T_G})$ is the subgroup of $X_*({\bf S})$ made out of one parameter subgroups annihilated by all (simple) roots of $\bf S$ in $\bf U$.  Thus $a_{\bf G}$ is the subspace of $a_{\bf S}$ orthogonal to all (simple) roots of $\bf S$ in $\bf U$. It has a natural complement $a^{\bf G}$ in $a_{\bf S}$ spanned by the (simple) coroots of $\bf S$ in $\bf U$. Dually, $a_{\bf G}^*$ is the subspace of $a_{\bf S}^*$ orthogonal to all (simple) coroots of $\bf S$ in $\bf U$, and has a complement spanned by the (simple) roots. On the other hand, there is a natural projection
\begin{equation}\label{eq:proj}
  p: a_{\bf Z} \rightarrow a_{\bf S},
\end{equation}
obtained by averaging the Galois orbits. More precisely, we have that $\Gamma_F$ acts on $a_{\bf Z}$, and the action factors through a finite Galois extension $F'$ of $F$, then, we can take
\[ p(z) = \dfrac{1}{\left| {\rm Gal}(F'/F) \right|}\sum_{\sigma \in {\rm Gal}(F'/F)} \sigma(z), \]
for $z \in a_{\bf Z}$.

Let ${\bf P}=\bf{MN}$ be a parabolic subgroup of $\bf G$ containing $\bf Z$, and $\Delta_M$  the subset of $\Delta$ made of simple roots of $\bf S$ in ${\bf M}\cap {\bf U}$. The previous considerations apply with $\bf M$ in place of $\bf G$, where
\[ a_{\bf S} = a_{\bf M} \oplus a^{\bf M}. \]
Dually, $a_{\bf M}^*$ is the subspace of $a_{\bf S}^*$ orthogonal to all coroots $\alpha^\vee$ for $\alpha$ in $\Delta_M$.

\subsection{Intertwining operators and Plancherel measures}\label{ss:intertPlan}

Let us recall some notations about intertwining operators using the notation of \cite{Re2010}. Let ${\bf P}={\bf MN}$ and ${\bf Q}={\bf MU }$ be  parabolic subgroups of $\bf G$ containing $\bf Z$. We denote by $\overline{\bf P}={\bf M\overline{N}}$ the opposite parabolic of $\bf P$. For an irreducible $PQ$-regular representation $\pi$ of ${\bf M}(F)$, there is a unique non-zero element (up to a non-zero constant) in $\operatorname{Hom}(\operatorname{Ind}_P^G\pi,\operatorname{Ind}_Q^G\pi)$, called the intertwining operator $J_{Q|P}(\pi)$, such that for every $f\in \operatorname{Ind}_P^G\pi$ with $\operatorname{Supp}(f)\cup \overline{PQ}\subset PQ$ [\emph{loc.\,cit.}, Proposition VII.3.4],
\[J_{Q|P}(\pi)(f)(1)=\int_{\overline{N}\cap U}f(u)du.\]

Making use of the identification \eqref{eq:rootspace}, we can meromorphically extend the intertwining operator to $\nu \mapsto J_{Q|P}(\pi_\nu)$, $\nu \in a^*_{\bf M}\otimes \mathbb C$ \cite[Th\'eor\`eme IV.1.1]{Wa2003}.

If $\pi$ is tempered and $\eta$ a character of $M$, positive with respect to $N$ and trivial on $M^1$, then $\eta \pi$ is $P\overline{P}$-regular [\emph{loc.\,cit.}, Lemme VII.4.1]. Thus, we can define in this situation the intertwining operator $J_{\overline{P}|P}(\eta\pi)$. 
The representation $\operatorname{Ind}_P^G(\eta\pi)$ has a Langlands quotient that is moreover the image of the intertwining operator $J_{\overline P|P}(\eta\pi)$  [\emph{loc.\,cit.}, Th\'eor\`eme VII.4.2]. 

Finally, suppose that $\pi$ is of discrete series, then $J_{P|\overline{P}}(\pi_\nu)\circ J_{\overline{P}|P}(\pi_\nu)$ is a scalar. By definition, as $\nu$ varies this defines a meromorphic function whose inverse is up to a constant the Plancherel measure or the Harish-Chandra function $\nu \mapsto \mu(\pi_\nu)$, $\nu \in a^*_{\bf M}\otimes \mathbb C$ \cite[Section V.2]{Wa2003}.

 \subsection{Local coefficients and $\gamma$-factors} \label{LS:lc:gamma}

The LS local coefficient is defined in the following setting. Let 
 $({\bf G},{\bf M})$ be a pair of quasi-split connected reductive groups such that $\bf M$ is a Levi component for a maximal parabolic subgroup $\bf P$ of $\bf G$ over the non-archimedean local field $F$. Assume that $\bf P$ is a maximal proper parabolic subgroup of $\bf G$ containing $
\bf B$, hence $\bf P$ corresponds to $\Delta - \{ \alpha \}$ for a choice of simple root $\alpha$. Let $\rho_P$ be the sum of all roots of $\bf S$ in $\bf N$; it is a character of $\bf M$, hence gives a central cocharacter of ${}^LM$. Let
\begin{equation}\label{widealpha}
    \widetilde{\alpha} = \dfrac{\rho_P}{\left\langle \rho_P, p(\underline \alpha^\vee) \right\rangle} \in a_{\bf M}^*,
\end{equation}
where $\underline \alpha$ is an absolute root in $a^*_{\bf T}$ restricting to $\alpha$, $\left\langle \cdot\, , \cdot \right\rangle$ is the perfect pairing between roots and co-roots, $\underline \alpha^\vee$ is the coroot associated to $\underline \alpha$, and $p(\underline \alpha^\vee)$ is obtained via the projection of \eqref{eq:proj} to $a_{\bf S}$.

We choose the splitting of $\bf G$ (giving a splitting of $\bf M$) so that $\pi$ is $\theta_{M,\psi}$-generic, where $\theta_{M,\psi}$ is the restriction to $U \cap M$ of the character $\theta_\psi$ of $U$. There is a natural long element $w_0$ in $W$ attached to $\bf M$, and a way to choose a lift $\tilde{w}_0$ to $G$, see Section 2.1 of \cite{LomLS}. For that choice, $\theta_\psi$ and $\theta_{M,\psi}$ are $\tilde{w}_0$ compatible, i.e.,
\begin{equation}
    \theta_{M,\psi}(u) = \theta_\psi(\tilde{w}_0 u \tilde{w}_0^{-1}), \ u \in U \cap M,
\end{equation}
and to $\pi$ is attached a local coefficient
\[ C_\psi(s,\pi) = C_{\theta_\psi}(s\widetilde{\alpha},\pi,\tilde{w}_0), \]
introduced by Shahidi in \cite{Sha1981} using the multiplicity one property of Whittaker functionals. The LS local coefficient $C_\psi(s,\pi)$ is a rational function of $t=q^{-s}$, see Theorem~1.3 of \cite{LomLS}, which is valid over any non-archimedean local field $F$.

Now we turn our attention to $\gamma$-factors. But first, we recall some facts about LS representations. In our setting, with ${\bf P} = {\bf M}{\bf N}$ a maximal parabolic of $\bf G$ corresponding to a simple root $\alpha$, the attached LS representations of ${}^LM$ in ${}^L\mathfrak{n}$ are naturally indexed by integers from $1$ to some positive integer $m = m(G,P)$, so that $r = r_i$ for some $i$. Specifically, $r_i$ is the representation of ${}^LM$ on the subspace of ${}^L\mathfrak{n}$ spanned by the root subspaces corresponding to roots $\beta^\vee$ of ${}^LM$ (i.e. coroots of M) with $\left\langle \widetilde{\alpha},\beta^\vee \right\rangle =i$. In other words, seeing $\rho_P$ as a central cocharacter of ${}^LM$, $r_i(\rho_P(x))$ is the scalar endomorphism of $r_i$ given by $x^{i \left\langle \rho_P, p(\underline \alpha^\vee) \right\rangle}$. In general, since $r_i$ acts trivially in central elements of ${}^LG$, we have that for any $\nu \in a^*_{\bf M}\otimes \mathbb C$, and again seeing it as a central cocharacter of ${}^LM$, we have that $r_i(\nu(x))$ is the scalar endomorphism of $r_i$ given by $x^{i \left\langle \nu, p(\underline \alpha^\vee) \right\rangle}$.  This implies that $i\langle \nu, p(\underline \alpha^\vee)\rangle$ can be recovered solely from $r$, hence it does not depend on the situation $M\subset G$ that gives rise to it.

Now, from (3.11) of \cite{Sha1990} and (5.1) of \cite{LomLS}, we recursively define $\gamma$-factors via the equation
\begin{equation}\label{C:prod:gamma}
    C_\psi(s,\pi) = c \prod_{i=1}^m \gamma(is,\pi,r_i,\psi),
\end{equation}
where $c$ is a monomial in $t=q^{-s}$, which does not affect the location of zeros and poles.

The $\gamma$-factors satisfy the following twisting formula, for every $\nu\in a_{\bf M}^*\otimes \mathbb C$, we have 
\begin{equation}\label{eq:gamma:shift}
    \gamma(s,\pi_\nu,r_i,\psi)= \gamma(s+i \left\langle \nu,p(\underline \alpha^\vee) \right\rangle ,\pi,r_i,\psi).
\end{equation}
Let us see why this formula holds up to a monomial in $t=q^{-s}$, which is enough for the location of zeros and poles. From the definition of Shahidi's local coefficient \cite{Sha1981}, one has \[C_\psi(s,\pi_\nu,\psi)=C_\psi(s+\langle\nu,p(\underline \alpha^\vee)\rangle,\pi,\psi).\]
Using formula \eqref{C:prod:gamma}, we get 
\[\prod_i\gamma(is,\pi_\nu,r_i,\psi)=A\prod_i\gamma(is+i\langle\nu,p(\underline \alpha^\vee)\rangle,\pi,r_i,\psi),\]
where $A$ is a non-zero constant. By induction, assume that for all indices $i$ except one, say $i_0$, 
\[\gamma(s,\pi_\nu,r_i,\psi)=A_i\gamma(s+i\langle\nu,p(\underline \alpha^\vee)\rangle,\pi,r_i,\psi).\]
We need to prove that 
\[\gamma(s,\pi_\nu,r_{i_0},\psi)=A_{i_0}\gamma(s+i_0\langle\nu,p(\underline \alpha^\vee) \rangle,\pi,r_{i_0},\psi),\]
for $A_{i_0}=A/\prod_{i\neq i_0}A_i$.
From the inductive nature of the definition of $\gamma$-factors, the assumption above is true, since we can interpret the shift $i\langle\nu,p(\underline \alpha^\vee) \rangle$ solely in terms on $r_i$ and not on the  particular situation of $M$ and $G$, as mentioned above. The shift is entirely dictated by the way that the maximal split torus of ${}^LM$ acts in $r_i$.

Given an LS representation $r$, we sometimes find it useful to introduce the cocharacter
\begin{equation}\label{eq:delta_r}
\delta_r = jp(\underline \alpha^\vee),
\end{equation}
where $r = r_j$, $j$ the respective nilpotency class. As we remarked above, we have that for $\nu \in a_{\bf M} ^*\otimes\mathbb C$, $\left\langle \nu, \delta_r\right\rangle$ depends only on $r$ and not in the situation $M\subset G$. 

\section{Tempered $L$-functions and standard modules conjectures}\label{s:temp}

In this section we prove Theorems~\ref{TLC} and \ref{SMC}, by showing that the characteristic zero proofs of Heiermann and Opdam (for Theorem~\ref{TLC}), and Heiermann and Mui\'c (for Theorem~\ref{SMC}), carry over to the case of positive characteristic.

As already recalled in the introduction, for the tempered $L$-function conjecture in characteristic 0, the method of Heiermann and Opdam \cite{HeOp2013} exploits results of Silberger \cite{Si1979} on the Plancherel measure, and of Heiermann \cite{He2004} on the location of discrete series. The problem is translated to a Hecke algebra defined by generators and relations, where the result for that algebra comes from the fact that it is a Hecke algebra occurring in groups for which the conjecture is already known. Now, the results of Silberger and Heiermann are valid in all characteristics, and the translation to an explicit Hecke algebra only necessitates basic facts on $L$-functions for cuspidal representations, due to Shahidi in characteristic 0 and extended to positive characteristic by the third-named author.

In characteristic zero, Heiermann and Mui\'c reduced the standard module conjecture to the tempered $L$-function conjecture in \cite{HeMu2007}, and our results on $L$-functions allow the exact same derivation in positive characteristic. Let us now proceed, using mostly the same notation as \cite{HeOp2013}, without recalling all of it; hence, the reader has to rely on Sections 3 to 5 of [\emph{loc.\,cit.}] and should consult their introduction for a good summary of notation and facts. We focus on the places where arguments, or references, have to be supplied when $F$ has characteristic $p$.

\subsection{}\label{LS:local:factors}
Our first and main objective is to prove Theorem~\ref{TLC}, where ${\bf P}={\bf MN}$ is a  parabolic subgroup of ${\bf G}$ containing ${\bf B}$. Throughout this section, we work with the theory of $L$-functions and local factors for an LS representation $r$ and a generic tempered irreducible smooth representation $\pi$ of $M$. The goal under these assumptions is to show that $L(s,\pi,r)$ is holomorphic on ${\rm Re}(s)>0$.

Recall that the LS method first produces $\gamma$-factors $\gamma(s,\pi,r,\psi)$, and then $L$-functions and $\varepsilon$-factors are derived from them. More precisely, the local $L$-function for tempered $\pi$ is defined as follows: write $\gamma(s,\pi,r,\psi)$ as a product of a monomial in $t=q^{-s}$ times a product of factors $1-a_jt$, divided by a product of factors $1-b_k t$, where none of the $a_j$'s are equal to any of the $b_k$'s; then 
\begin{equation}\label{gamma:fe}
    L(s,\pi,r) = \dfrac{1}{\displaystyle \prod_j (1-a_jt)}.
\end{equation}
With this setup, Theorem~\ref{TLC} for $L(s,\pi,r)$ says that each $\left| a_i \right| \leq 1$. We observe that changing $\psi$ modifies the $\gamma$-factor by a monomial on $t$, so the $L$-functions are independent of $\psi$.

Another important factor produced by the LS method is the $\varepsilon$-factor, that can be defined via the following relation
\begin{equation}
    \gamma(s,\pi,r,\psi) = \varepsilon(s,\pi,r,\psi) \dfrac{L(1-s,\tilde{\pi},r)}{L(s,\pi,r)}.  \label{eq:gammaepsilonL-func}
\end{equation}
 Note that the functional equation for $\gamma$-factors 
\[\gamma(s,\pi,r,\psi) \gamma(1-s, \tilde{\pi},r,\psi^{-1})=1,\]
 implies the functional equation of $\varepsilon$-factors
 \[\varepsilon(s,\pi,r,\psi) \varepsilon(1-s, \tilde{\pi},r,\psi^{-1})=1.\]
 Furthermore, assuming that $L(s,\pi,r)$ is holomorphic for ${\rm Re}(s)\geq1/2$ (which it is implied by the Theorem \ref{TLC}), then $\varepsilon$ is a monomial in $t=q^{-s}$. In other words the product of the factors $(1-b_jt)^{-1}$ is $L(1-s,\tilde{\pi}, r)$
up to a monomial in $t$.

 Let us go back to the our goal.  In terms of $\gamma$-factors, saying that $L(s,\pi,r)$ is holomorphic for ${\rm Re}(s)>0$ is equivalent to saying that $\gamma(s,\pi,r,\psi)$ is non-zero for ${\rm Re}(s)>0$, including poles. Therefore, to prove Theorem \ref{TLC}, one reduces to the case where $\bf P$ is maximal proper by the multiplicativity property of $\gamma$-factors, identity (3.13) of \cite{Sha1990} and Theorem 5.1 (iv) of \cite{LomLS}. Moreover, using the same property and that a tempered representation is a subrepresentation of representation parabolically induced from discrete series representation, we may assume if convenient that $\pi$ is a discrete series representation.

\subsection{} The first basic result is that Theorem~\ref{TLC} holds when $\pi$ is unitary cuspidal. More precisely, the following may be found in Section 7 of \cite{Sha1990} and Section 5 of \cite{LomLS}.

\begin{proposition}\label{prop:L:cusp}
 Assume that $\pi$ is (unitary) cuspidal. Then $L(s,\pi,r_i) = 1$ if $i>2$. For $i=1, 2$, the inverse $L$-function can be written as a (possibly empty) product
 \begin{equation*}
     L(s,\pi,r_i)^{-1} = \prod (1-a_jt),
 \end{equation*}
where the $a_j$'s are complex numbers of absolute value $1$. In particular $L(s,\pi,r_i)$ has no pole on ${\rm Re}(s)>0$. Only one of $L(s,\pi,r_1)$ and $L(s,\pi,r_2)$ may have a pole at $s=0$. Moreover, this occurs exactly in the following situation: $\tilde{w}_0(M)=M$, $\tilde{w}_0(\pi) \simeq \pi$ and ${\rm Ind}_P^G(\pi)$ is irreducible.
\end{proposition}

Note the consequence used in \cite{HeOp2013}: for real $s$, there is only one index $i=1$ or $2$ such that $L(s,\pi,r_i)$ has a pole on the real axis, i.e. $\gamma(s,\pi,r_i,\psi)$ has a zero on the real axis, and that pole is at $s=0$. Indeed, if $s$ is real then $t=q^{-s}$ is real positive. And, we know that each $a_j$ is of absolute value $1$, then $1-a_jt$ is $0$ precisely when $a_j=1$ and $t=1$; in particular, when $s=0$. 

Note also that another consequence  of this result and the discussion after equation \eqref{eq:gammaepsilonL-func} is that, when $\pi$ is unitary cuspidal, the $\varepsilon$-factor is a monomial in $t=q^{-s}$ and the $\gamma$-factor verifies
\begin{equation}\label{gamma:e:L:cusp}
   \gamma(s,\pi,r,\psi) = \varepsilon(s,\pi,r,\psi) \dfrac{L(1-s,\tilde{\pi},r)}{L(s,\pi,r)},
\end{equation}
where $\Tilde\pi$ is the contragredient of $\pi$.

\subsection{}We are ready to prove Theorem \ref{TLC}. As we mention in Section \ref{LS:local:factors} we may and do assume that $\bf P$ is maximal in $\bf G$ and that $\pi$ is a discrete series. Since $\pi$ is generic, its cuspidal support is also  generic and there is parabolic subgroup ${\bf P}_1={\bf M}_1{\bf N}_1$ of $\bf G$ contained in $\bf P$, a  $\Theta_{M_1,\psi}$-generic unitary cuspidal representation $\sigma$ of $M_1$ and a character $\chi$ of $M_1$ with positive real values and trivial on $M_1^1$ such that $\pi$ is a component of $\operatorname{Ind}_{P_1\cap M}^N(\chi \sigma)$. In this situation we write $\chi=\chi_{\nu_\pi}$ for some element $\nu_\pi$ in $a_{{\bf M}_1}^{{\bf M},*}$. Heiermann and Opdam show that one can chose $\sigma$ and $\chi$ so that $\pi$ us a subrepresentation of $\operatorname{Ind}^M_{P_1\cap M}(\chi\sigma)$. Then they apply result of  Silberger \cite{Si1979}, valid in all characteristic, to study the element $\nu_\pi$, and they use previous results of Heiermann \cite{He2004}, again valid in all characteristic, to show that $\nu_\pi$ is a ``residue point'' of the Harish-Chandra function $\nu \mapsto \mu(\sigma_\nu)$. Directly form the definition of intertwining operator and local coefficient, we deduce in the same way that Shahidi does in \cite{Sha1981}, now in characteristic $p$, that \[\mu(\sigma_\nu)=C_\psi(\tilde w_0(\nu),\tilde w_0(\sigma),\tilde w_0^{-1})C_\psi(\nu,\sigma,\tilde w_0).\] 

Then, we again follow Heiermann and Opdam \cite{HeOp2013}, who use the relation between local coefficient and $\mu$ functions, and the multiplicativity property of the LS $\gamma$-factors, to get their Proposition 3.1. All the facts that they use being true when $F$ has characteristic $p$, thus this proposition is also valid in this case. In particular, there are numerical data $\epsilon_{\overline{\beta}}$ attached to $\sigma$ and $\chi$. 

Secondly, Theorem 4.2 in \cite{HeOp2013} concerns the abstract $C$-functions and $\gamma$-factors attached to a root system and some numerical data that we just encountered. More precisely, it gives a condition on these numerical data in order for $C$-functions to be holomorphic on the real numbers less than $0$ and $\gamma$-factors to be non-vanishing on the real numbers greater than $0$.  Its proof uses split reductive groups over $F$ (which in [\emph{loc.\,cit.}] has characteristic $0$), but those groups only enter through the number $q$, so we can instead use a $\mathfrak{p}$-adic field $F'$ with $q_{F'} = q$, to obtain the theorem in characteristic $p$.

\subsection{} 

Finally, let us say why Theorem 5.1 of \cite{HeOp2013} is valid in our case, yielding our Theorem~\ref{TLC}. The point now is to establish the numerical criterion in Theorem 4.2 of [\emph{loc.\,cit.}], which concerns numbers $\epsilon_\beta$ attached to the roots $\beta$ in the root system $\Sigma$ of $\bf G$. Let $\Sigma_1$ be the root system of ${\bf M}_1$. In Section 4 of [\emph{loc.\,cit.}], it is proved, essentially via combinatorics of root systems valid in all characteristics, that it is enough to treat the case where $\Sigma_1$ has corank 2 in $\Sigma$. One exception is for a local-global argument where Proposition 5.1 of \cite{Sha1990}, valid only when $F$ has characteristic $0$, is invoked. For $F$ of positive characteristic, one has to replace that reference with the main result of \cite{GaLo2018}. In Section 6 of \cite{HeOp2013}, the numerical criterion is proved by inspection of all the corank $2$ situations, after some standard reductions; that also uses only combinatorics of root systems, so all is valid for our $F$.

\subsection{} When $F$ has characteristic $0$, Theorem \ref{SMC} is deduced from Theorem \ref{TLC} by the results of Heiermann and Mui\'c \cite{HeMu2007}. Their main input is again the result of Heiermann on the location of discrete series, valid in all characteristics. Given the theory of $C$-functions and $\gamma$-factors in positive characteristic, as obtained following \cite{LomLS}, their proof goes through entirely and yields our Theroem \ref{SMC}.

\subsection{} Let us assume that $\bf G$ is split. Then the method of close local fields is available by \cite{Ga2015}. With that method, Lomelí establishes our Theorem \ref{TLC} in Section 5 of \cite{Lom2019}, and Theorem \ref{SMC} follows as above. In Theorem 5.6 of [\emph{loc.\,cit.}], a result on the generic unitary dual is transferred from the characteristic $0$ case \cite{LaMuTa2004}. We point out that it should be possible to transfer the full classification of Lapid, Mui\'c and Tadi\'c in this way; indeed, the criteria determining which case of the classification applies can be expressed in terms of $L$-factors, Section 2.4 of [\emph{loc.\,cit.}], and those are preserved through the comparison of close local fields as in Proposition 5.4 of \cite{Lom2019}.

\section{On Langlands-Shahidi Local $L$-functions}\label{s:L-function}
In this section, we carefully explain how to extract the LS $L$-function $L(s,\pi,r)$ from the LS factor $\gamma(s,\pi,r,\psi)$ following Shahidi's recipe, outlined in his fundamental paper \cite{Sha1990}. As observed in \cite{LomLS}, while passing from tempered to general generic representations via Langlands classification one needs to ensure that no cancellations occur in the resulting products of rational functions in order to properly define LS $L$-functions. A holomorphy condition on tempered LS $L$-functions for $\operatorname{Re}(s) \geq 1/2$ bridges this gap in the definition, which is more than ensured by the validity of the tempered $L$-function conjecture.

\subsection{}
We are in the situation of the main text and we keep the notation; in particular, $\bf G$ is a quasi-split connected reductive group over $F$. We let $r$ be an LS-representation of ${}^LG$ and set $t=q^{-s}$. Fix a generic irreducible representation $\pi$ of $G$. Then the $\gamma$-factor $\gamma(s,\pi,r,\psi)$ is defined, it is a rational function of $t$. The $L$-functions and $\varepsilon$-factors are derived from $\gamma$-factors, first when $\pi$ is tempered, then when $\pi$ is essentially tempered, and finally in general, using the Langlands classification. 

The following discussion helps us explain in Section \ref{s:unr} why for a $K$-unramified generic irreducible representation $\pi$, the $L$-function $L(s,\pi,r)$ is the same as that attached to the Langlands parameter of $\pi$, the equality being true by design for $\gamma(s,\pi,r,\psi)$.

\subsection{}\label{subsec:LSreview}
Let us start with the case where $\pi$ is tempered. We recall the notation introduced in Section \ref{LS:local:factors}. The factor $\gamma(s,\pi,r,\psi)$ is written as product of a monomial in $t$ times a product of factors $1-a_jt$, divided by a product of factors $1-b_k t$, where none of the $a_j$'s are equal to any of the $b_k$'s; and thus the $L$-function is given by
\begin{equation}
    L(s,\pi,r) = \dfrac{1}{\displaystyle \prod_j (1-a_jt)}.
\end{equation}

\subsection{}\label{A:shift}
Next is the case where $\pi$ is \emph{essentially tempered}, that is, of the form $\chi \pi_0$ where $\pi_0$ is tempered and $\chi$ is a character of $G$. We can take $\chi$ to have positive real values, since twisting by a unitary character does not affect temperedness, and then $\chi$ and $\pi_0$ are unique. The recipe of Shahidi is then to ``proceed by analytic continuation'', and that we have to explain because the factor $\gamma(s,\chi \pi_0,r,\psi)$ certainly depends on the character $\chi$ in an analytic, even rational, manner, but we know no way to consistently define a numerator and a denominator for a function of several real or complex variables, as is done above for a rational function of $t$.

As we review in Section \ref{LS:lc:gamma}, for any root $
\beta^\vee$ of ${}^LM$ that is in the space of $r$, there is a positive integer $j$ such that $\left\langle \widetilde{\alpha}, \beta^\vee \right\rangle = j$, and we let $\delta_r$ be as in \eqref{eq:delta_r}.

We have (up to a monomial) the following formula 
\begin{equation}
    \gamma(s,\chi\pi_0,r,\psi) = \gamma(s+ \left\langle  \chi,\delta_r \right\rangle ,\pi_0,r,\psi).
\end{equation}
We use then the same formula for the 
$L$-function, that is we let 
\[
L(s,\chi\pi_0 ,r)=L(s + \left\langle  \chi,\delta_r \right\rangle,\pi_0,r).
\] 
That formula holds in particular when $\chi$ is a unitary character (so that $\chi\pi_0$ is still tempered). That is unambiguous and the equality
\begin{equation} L(s,\chi\pi,r)=L(s+ \left\langle  \chi,\delta_r \right\rangle,\pi,r)
\end{equation}
also holds when $\pi$ is essentially tempered.

\subsection{}\label{subsec:LS-Lgeneral}
Finally, in the general case, we write $\pi$ as the Langlands quotient of a standard module, say ${\rm Ind}_P^G (\chi \tau)$ (by Theorem \ref{SMC}, $\pi$ is the full standard module, which is irreducible). Then the $\gamma$-factor of $\pi$ is given as follows: write the restriction of $r$ to ${}^LM$ (where ${\bf P} = {\bf MN}$) as a direct sum of irreducible representations $r_{M,i}$. Each is an LS-representation, and the multiplicativity property of $\gamma$-factors says that
\begin{equation}\label{eq:gammamult}
\gamma(s,\pi,r,\psi) = \prod_i \gamma(s,\chi\tau,r_{M,i},\psi).
\end{equation}
We then define
\begin{equation}\label{eq:Lmult}
 L(s,\pi,r) = \prod_i L(s,\chi\tau,r_{M,i}).
\end{equation}
\begin{remark}
The restriction of $r$ to ${}^LM$ is semisimple. Indeed, its restriction to $\widehat{M} = {}^LM^\circ$ certainly is semisimple because $\widehat{M}$ is a complex reductive group and $r$ is algebraic, and then its restriction to ${}^LM$ is semisimple because $r(\widehat{M})$ has finite index in $r({}^LM)$, see \cite{He2001}.
\end{remark}

\begin{remark}
Let $\tau$ be a tempered irreducible representation of $M$, and assume that $\pi$ is the generic irreducible component of ${\rm Ind}_P^G(\tau)$. Then, we have the multiplicativity property of \eqref{eq:gammamult}.  Because of the tempered $L$-function conjecture (Theorem \ref{TLC}), the same property as \eqref{eq:Lmult} holds for $L$-functions.
In fact, for that property it is enough to know that the functions $L(s, \tau, r_{M,i})$ and $L(s,\tilde{\tau},r_{M,i})$ have no pole for ${\rm Re}(s) \geq 1/2$. Indeed, that is sufficient to imply that there is no cancellation between the polynomials in $t$ of
\[ \dfrac{1}{L(s,\tau,r_{M,i})} \quad \text{and} \quad  \dfrac{1}{L(1-s,\tilde{\tau}, r_{M,i'})}, \]
for different $i,i'$.
\end{remark}

\begin{remark}
For $L$-functions, the multiplicativity works when the induced representation is the standard module, but not necessarily in general. Indeed, already in the simple example of ${\bf G} = {\rm GL}_2 \times {\rm GL}_1$ seen as a block diagonal Levi subgroup of ${\rm GL}_3$. In this case, for an irreducible generic representation of ${\rm GL}_2(F)$ and a character $\chi$ of $F^\times$, the $L$-function $L(s,\pi \otimes \chi, r)$ is simply the Jacquet-Langlands (or Godement-Jacquet) $L$-function $L(s,\chi^{-1}\pi)$ with $r$ given by
the standard representation of ${}^LG={\rm GL}_2(\mathbb{C})$. If we take for $\pi$ the Steinberg representation $\rm St$ of ${\rm GL}_2(F)$ and for $\chi$ the trivial character, then $\rm St \otimes 1$ is the irreducible generic component of a principal series ${\rm Ind}_B^G (\tau)$, where ${\bf B}= \bf{TU}$ is the Borel subgroup of $\bf G$, and $\tau$ is essentially tempered but not tempered, but the $L$-function $L(s,{\rm St})$ is not equal to $L(s,\tau, r') L(s,\tau,r'')$ where $r'$ and $r''$ are the two characters components of the restriction of $r$ to ${}^LT$. Indeed, there is cancellation between the denominators and numerators of the two factors $\gamma(s,\tau,r',\psi)$ and $\gamma(s,\tau,r'',\psi)$: concretely $\bf T$ is ${\rm GL}_1^3$ and $r'$ is the character $(x_1,x_2,x_3) \mapsto x_1x_3^{-1}$ of $(\mathbb{C}^\times)^3$, whereas $r''$ is the character $(x_1,x_2,x_3) \mapsto x_2x_3^{-1}$. As a representation of $T=(F^\times)^3$, $\tau$ is the character $(a_1,a_2,a_3) \mapsto \left| a_1a_2^{-1} \right|_F^{1/2}$, and we have \[ \dfrac{1}{L(s,\tau,r')} = 1-q^{-1/2}t \quad \text{and} \quad \dfrac{1}{L(s,\tau,r'')}=1-q^{1/2}t. \]
On the other hand, we have
\[ \dfrac{1}{L(s,\tilde{\tau}, r')} = 1-q^{1/2}t \quad \text{and} \quad \dfrac{1}{L(s,\tilde{\tau}, r'')} = 1-q^{-1/2}t, \]
so
\[ \dfrac{1}{L(1-s,\tilde{\tau}, r')} = 1-q^{3/2}t^{-1} \quad \text{and} \quad \dfrac{1}{L(1-s,\tilde{\tau},r'')} = 1-q^{1/2}t^{-1}, \]
and we conclude
\[ \dfrac{1}{L(s,{\rm St})} = 1-q^{1/2}t. \] 
\end{remark}

\section{On unramified representations  and its $L$-functions}\label{s:unr}

In this section we first review the basic about unramified representations and give a description of their standard modules. Then, we calculate the LS $L$-functions for unramified generic irreducible representation $\pi$, following the discussion in Section \ref{s:L-function}. Indeed, we are going to prove that it is equal to what to expect, i.e.  the $L$-function $L(s,\pi,r)$ attached to the unramified Langlands parameter of $\pi$.

\subsection{Unramified  characters}
 A character of $G$ is (often) called unramified if it is trivial on $G^1$. Such characters of $G$ are smooth. However, because of the following remark, we prefer to speak of characters of $G$ that are trivial on $G^1$, and not of unramified characters of $G$.

\begin{remark}\label{rmk:G1-unr}
Assume that $\bf G$ is unramified, that is quasi-split and split over an unramified extension. Let $K$ be a hyperspecial maximal compact subgroup of $G$. An irreducible representation of $G$ is $K$-unramified (or $K$-spherical) if it has non-zero vectors fixed by $K$. A smooth character of $G$, seen as a representation of dimension $1$, is $K$-unramified exactly when it is trivial on $K$. Clearly a character trivial on $G^1$ is $K$-unramified, but the converse is not necessarily true, even for split groups. For ${\bf G}={\rm PGL}_2$, the special maximal compact subgroups of $G$ are all hyperspecial; they are conjugates of $K={\rm PGL}_2(\mathcal O_F)$; and, the subgroup $G'$ they generate is made out of elements with determinant a unit in $F^\times/(F^\times)^2$, whereas $G^1=G$. Thus, the unique non-trivial character of $G$ trivial on $G'$ is $K$-unramified for any choice of hyperspecial maximal compact subgroup $K$ of $G$, but not trivial on $G^1$.
\end{remark}

When $G$ is unramified, we specify $K$-unramified (or $K$-spherical) representations, unless there is no ambiguity, or the choice of $K$ has no influence.

\subsection{Special maximal compact subgroups and unramified representations}\label{subsec:unrmaxcpt}

Let $K$ be a special maximal compact subgroup of $G$ corresponding to a special point of the apartment of $S$ in the building of $G$. Then $Z \cap K$ is the unique maximal compact subgroup of $Z$, equal to the group $Z^1 = \ker H_Z$. There is thus no ambiguity in speaking of unramified characters of $Z$ for characters trivial on $Z^1$ (see Remark \ref{rmk:G1-unr}). If $\chi$ is such a character, the parabolically induced representation
\[ {\rm Ind}_B^G(\chi)\] is usually called an unramified principal series; it has a unique $K$-unramified
component. Writing $\pi(\chi)$ for the isomorphism class of that component we get a bijection 
\begin{equation}\label{eq:pi:chi}
    \chi \mapsto \pi(\chi)
\end{equation}
from $W$-orbits of unramified characters of $Z$ onto isomorphism classes of $K$-unramified irreducible representations of $G$.

Let $K'$ be another special maximal compact subgroup  of $G$. Assume first that $K'$ is $gKg^{-1}$, a conjugate  of $K$ in $G$. A vector $v$ in a representation $V$ of $G$ is fixed by $K$ if and only if $g\cdot v$ is fixed by $K'$, so that the above classification is true also for $K'$, and the $K'$-unramified component is the same as the $K$-unramified one. For a general $K'$, the corresponding special point is the translate, by an element $g$ of $G$, of a special point in the apartment attached to $S$, so $K'$ is conjugate via $g$ to a special maximal compact subgroup in good position with respect to $S$, and consequently the classification also holds for $K'$.

\begin{remark}
Without further assumptions on $\bf G$, it is not true that special maximal compact subgroups are conjugate under ${\bf G}_{ad}$, nor even under $\operatorname{Aut}(G)(F)$, as the example of ${\bf G}={\bf U}_{2,1}$ already shows (cf. Appendix~\ref{SL:SU}). When $\bf G$ is unramified, however, hyperspecial maximal compact subgroups of $G$ are conjugate under $G_{ad}$ \cite[Section 2.5]{Ti1979}.
\end{remark}

\subsection{Langlands quotient and unramified representations} Assume now that $\bf G$ is unramified, and let $K$ be a hyperspecial maximal compact subgroup of $G$, corresponding to a hyperspecial point in the apartment of $S$ in the building of $G$. Since $\bf G$ is quasi-split, ${\bf Z} = {\bf Z_G}({\bf S})$ is in fact a torus and thus we will denote it by $\bf T$, as in Section \ref{subsec:Kazhdan}. Given an unramified character $\chi$ of $T$, let $\pi(\chi)$ be the corresponding $K$-unramified isomorphism class of \eqref{eq:pi:chi}.  In this section we want to express $\pi(\chi)$ in  terms of the Langlands classification.  This will be used in Section \ref{subsec:unrLS} to relate the LS $L$-function of $\pi(\chi)$ with the $L$-function of its Galois parametrization. We are grateful to Waldspurger for discussions on that question, and the essence of the argument.

Given an unramified character $\chi$ of $Z$, we write
\[ \chi=|\chi|\chi^u, \]
so that $\chi^u$ is a unitary character of $T$. Since, for $w \in W$ we have $|\chi^w|=|\chi|^w$ and $\pi(\chi)$ only depends on the $W$-orbit of $\chi$,  we may (and do) assume that $|\chi|$ is in the closed positive Weyl chamber, that is corresponds to $\nu \in a_{\bf T}^*$ with $\langle \nu,\alpha^\vee\rangle$ non-negative for every simple root $\alpha \in \Delta$ (note that $\bf S$ is the maximal $F$-split subtorus of $\bf T$, so that $a_{\bf T}^*$ identifies with $a_{\bf S}^*$).

\begin{proposition}\label{prop:extension}
Let $\chi$ be an unramified character of $T$. Let $\theta$ be the set of simple roots $\alpha$ in $\Delta$ such that  $\langle\nu,\alpha^\vee\rangle=0$, and let ${\bf P}={\bf MN}$ be the parabolic subgroup of $\bf G$ containing $\bf B$ such that $\theta$ is the set of the simple roots of $\bf S$ in $\bf M$. Then there is a unique character $\eta$ of $M$, with positive real values, restricting to $|\chi|$ on $T$. The character $\eta$ is positive with respect to $N$.
\end{proposition} 

\begin{proof}
 As we saw in Section \ref{subsec:orthogonal}, the orthogonal complement in $a_{\bf T}^*$ of the $\alpha^\vee$ for $\alpha\in \theta$  is $a_{\bf M}^*$, so an element  $\nu \in a_{\bf T}^*$ such that $\langle\nu,\alpha^\vee\rangle= 0$ for $\alpha \in \theta$ corresponds (uniquely) to a character $\eta$ of $M$ with positive real values, which necessarily restricts to $\chi$ on $T$. 
Since $\langle\nu,\alpha^\vee\rangle$ is positive for any simple root $\alpha \not \in \theta$, the character $\eta$ is positive with respect to $N$. That proves Proposition \ref{prop:extension}.
\end{proof}

Let us recall the intertwining operators introduced in Section \ref{ss:intertPlan}. Let ${\bf P}={\bf MN}$ be a parabolic subgroup of $\bf G$ containing $\bf B$ and $\overline{\bf P}$ the opposite parabolic of $\bf P$. For a tempered irreducible representation $\tau$ of ${\bf M}(F)$ and $\eta$ a positive with respect to $N$ unramified character of $M$,  we consider the intertwining operator given by 
\[J_{\overline P|P}(\eta\tau): \operatorname{Ind}_P^G(\eta\tau) \to \operatorname{Ind}_{\overline P}^G(\eta\tau).\]
Recall also that the representation $\operatorname{Ind}_P^G(\eta\tau)$ has a Langlands quotient that is moreover the image of the intertwining operator $J_{\overline P|P}(\eta\tau)$ (See again Section \ref{ss:intertPlan}).

\begin{proposition}\label{prop:unrtemp}
  Let $\tau=\pi_M( \chi^u)$ be the $K\cap M$-spherical irreducible component of $\operatorname{Ind}_{B\cap M}^M(\chi^u)$. Then $\tau$ is a tempered representation of $M$ and $\pi(\chi)$ is the Langlands quotient of $\operatorname{Ind}_P^G(\eta\tau)$, where $\eta$ is as in Proposition \ref{prop:extension}.
\end{proposition}

Before we prove Proposition \ref{prop:unrtemp}, we establish the following additional result.
\begin{proposition}\label{prop:intertnonzero}
The operator $J_{\overline P|P}(\eta\tau)$  is not zero on $K$-fixed vectors in $\operatorname{Ind}_P^G(\eta\tau)$.
\end{proposition}

\begin{proof}

We want to use the result in \cite{Ca1980}. So we start with some translation to that setting. First, let $B_1=(B\cap M)\overline{N}$ and note that $\eta\tau$ and $\chi$ are $P\overline{P}$ and $BB_1$-regular, respectively. We have the following commutative diagram \cite[Proposition VII.3.5 (iii)]{Re2010}
\[\begin{tikzcd}
\operatorname{Ind}_P^G(\eta \tau) \arrow[r,hookrightarrow]\arrow[d,"J_{\overline{P}|P}(\eta \tau)"'] & \operatorname{Ind}_{B}^G(\chi) \arrow[d,"J_{B_1|B}(\chi)"]  \\
\operatorname{Ind}_{\overline{P}}^G(\eta \tau) \arrow[r,hookrightarrow] &  \operatorname{Ind}_{B_1}^G(\chi).
\end{tikzcd}\]
So it is enough to study the operator $J_{B_1|B}(\chi)$.

Now, we translate $J_{B_1|B}(\chi)$ to the setting of \cite{Ca1980}. Since we are supposing that ${\bf G}$ is unramified, the smooth integral $\mathcal O_F$-model $\mathcal G$ such that $\mathcal G(\mathcal O_F)=K$ is reductive. This allows us to define $\overline N(\mathcal O_F)$. Now, by following [\emph{loc.\,cit.}], we choose our intertwining operator, defined in Section \ref{ss:intertPlan}, such that
\[ J_{B_1|B}(\chi)(f)(1)=\int_{\overline N(\mathcal O_F)}fdu=1, \]
where $f$ has support in $B\overline N(\mathcal O_F)$ and $f(bn_1)=\chi(b)$, for every $b\in B$ and $n_1\in \overline{N}(\mathcal{O}_F)$.  

Next, let  $w^G$ be  the longest element in $W$ and $\Tilde w$ be a representative of $w^Gw^M$ in $K$. Now let us consider
\[T(\Tilde w,\chi):=J_{B|B_1}(\Tilde w(\eta\tau))\circ t(\Tilde w)\colon \operatorname{Ind}^G_{B}(\chi)\to \operatorname{Ind}_{B_1}^G(\Tilde w(\chi)) \to \operatorname{Ind}^G_{B}(\Tilde w(\chi)),\] where $(t(\Tilde w)f)(g)=f(\Tilde w^{-1}g)$ and  $\Tilde w(\chi)(g)=\chi(\Tilde w^{-1}g\Tilde w)$. This variant is the one constructed in [\emph{loc.\,cit.}, Section 3]. For a $W$-regular character $\chi$ one can compute the effect of $T(\tilde{w},\chi)$ on $K$-fixed vectors and shows it is given by multiplication by a non-zero constant [\emph{loc.\,cit.}, Theorem 3.1]; more precisely, the constant is \[c_{\tilde{w}}(\chi)=\prod_{\alpha>0,w\alpha<0}\frac{(1-q_{\alpha/2}^{-1/2}q_\alpha^{-1}\chi(a_\alpha))(1+q_{\alpha/2}^{-1/2}\chi(a_\alpha))}{1-\chi(a_\alpha)^2},\]
where  $q_{\alpha/2},q_\alpha \in \mathbb Q_{>0}$ and $a_\alpha\in Z/Z^1$ are certain structure data coming from $G$. The product is over non-multipliable positive roots $\alpha$ of $S$ such that $w \alpha$ is negative, that is non-multipliable roots in $\mathfrak{n}$. The quantities $a_\alpha$, $q_\alpha$ and $q_{\alpha/2}$  are rather attached to  the affine root system attached to $G$, cf. [\emph{loc.\,cit.}, p.\,389, line 2 and p.\,394 line 1], and we may as well use non-divisible roots, which we do.
 From the proof of Theorem 3.1 in [\emph{loc.\,cit.}], the evaluation of the quantities $q_{\alpha/2},q_\alpha$ and $a_\alpha$ can be reduced to the case where $G$ is semisimple of rank one.
More precisely, for each $\alpha$ we have such a subgroup $G_\alpha$ of $G$, and the quantities 
depend only on the simply connected cover $\tilde{G}_\alpha$ of $G_\alpha$. 
Assume first that $\alpha$ is a simple root. Then the hyperspecial point $x_0$ of the apartment of $G$ giving rise to $K$ is hyperspecial also for $G_\alpha$ (the apartment of $G_\alpha$ being a line through $x_0$), so that $K \cap G_\alpha$ is a hyperspecial subgroup of $G_\alpha$. We are then reduced to the case of $G$ to be ${\rm SL}_2$ or the unramified quasi-split ${\rm SU}_{2,1}$ studied in Appendix~\ref{SL:SU}. 

In the first case $q_\alpha=q_F$ , $q_{\alpha/2}=1$ and $a_\alpha$ can be taken as ${\rm diag}(\varpi, \varpi^{-1})$ where $\varpi$ is a uniformizer of $F$ , cf. [\emph{loc.\,cit.}, Remarks 1.1 and 3.3].  In the second case, $a_\alpha$ can be taken to be ${\rm diag}(\varpi,1,\varpi^{-1})$, Then formula (15) of [\emph{loc.\,cit.}] gives $(q_\alpha)^2q_{\alpha/2}=q_F^4,$ whereas formula (14) yields 
$q_{\alpha/2}q_\alpha=q_F^3$.  Consequently $q_\alpha=q_F$ and $q_{\alpha/2}=q^2_F$.

Now, revert to the general case where $\alpha$ is not necessarily simple. Note that $\alpha$ can be 
conjugated by an element $w$ of the Weyl group to a simple root $\beta$.  We may choose $w$ in $K$, so that $w$ fixes the vertex $x_0$. The previous considerations then apply, and in particular $q_\alpha>1$, $q_{\alpha/2}\geq1$, and since the action of $a_\alpha$ on the root space attached to $\alpha$ is contracting, we get that $|\chi(a_\alpha)|_F<1$. That gives what we wanted.\end{proof}

\begin{proof}[Proof of Proposition \ref{prop:unrtemp}]
Since $\chi^u$ is unitary, $\tau$ is indeed tempered. Because $\eta$ is positive with respect to $N$, $\operatorname{Ind}_P^G(\eta\tau)$ does have a Langlands quotient and it is the image of the intertwining operator. 
Hence is $\pi(\chi)$. \end{proof}

\begin{remark}\label{rmk:specialdiscrete}
An important consequence of the claim is that $\pi(\chi)$ is essentially square integrable only if $G$ is a torus. When $\bf G$ is simple simply connected, that can be deduced from \cite[Theorem 5.1.2]{Mac1971}. 
Even for unramified $\bf G$ this fact does not necessarily hold if we take a special maximal compact subgroup $K$ instead of a hyperspecial one. That follows from §5.2 in [loc.\,cit.], and seems to have been noticed first in \cite{Mat1969}.
In fact those exceptional examples were the reason for the intervention of hyperspecial subgroups, rather than special subgroups, in the definition of the L-packet attached to an unramified parameter for unramified G \cite[p. 45]{BCor1979}.
\end{remark}

\subsection{Unramified Langlands-Shahidi L-functions} \label{subsec:unrLS} As in the previous section, assume now that $\bf G$ is unramified, and let $K$ be a hyperspecial maximal compact subgroup of $G$, corresponding to a hyperspecial point in the apartment of $S$ in the building of $G$. Moreover, suppose now that $\pi(\chi)$, the isomorphism class constructed in Section \ref{subsec:unrmaxcpt},  is generic. 

Now, let $r$ be an LS-representation of ${}^LG$. Then $\chi$ corresponds to an unramified homomorphism 
\[\sigma(\chi)\colon W_F\to  {}^LT,\] and \[\gamma(s, \pi(\chi), r, \psi)=\gamma(s, r \circ \iota\circ\sigma (\chi), \psi)\] by design, where $\iota$ is the inclusion of ${}^LT$ into ${}^LG$, so that $\iota \circ  \sigma(\chi)$ is the Langlands parameter (a homomorphism of $W_F$ to ${}^LG$) attached to $\pi(\chi)$.
The following proposition is important in Shahidi’s approach; it is implicit in \cite{Sha1990}.

\begin{proposition}One also has
\[L(s, \pi(\chi), r)=L(s, r \circ \iota \circ \sigma(\chi)).\]
\end{proposition}  

\begin{proof}
 We follow the definition of $L$-functions in Section \ref{s:L-function}, and proceed along the Langlands classification applied to $\pi(\chi)$, as explained in Section \ref{subsec:unrmaxcpt}, of which we use the notation.

The first thing is to prove
\[ L(s, \pi_M(\chi^u), r')=L(s, r'\circ \iota_M \circ \sigma(\chi^u)) \]
for any irreducible component $r'$ of the restriction of $r$ to ${}^LM$, where  $\iota_M$ is the inclusion of ${}^LT$ in ${}^LM$. The equality is true for the $\gamma$-factor on both sides, by the very construction of $\gamma$-factors. Moreover $r'\circ\iota_M \circ \sigma(\chi^u)$ is a direct sum of unitary characters of $W_F$, so there can be no cancellation between numerator and denominator of the $\gamma$-factors of those unitary characters. The equality for the $L$-functions follow from the equality of $\gamma$-factors, and the definition of the LS $L$-functions.

The second stage consists in proving
\[ L(s, \eta\pi_M(\chi^u), r')=L(s, r'\circ \iota_M \circ \sigma(\chi)). \]
In Section \ref{A:shift} above  we showed that  
\[ L(s, \eta\pi_M(\chi^u), r')=L(s+\langle \eta,\delta_{r'}\rangle, \pi_M(\chi^u), r'). \]
On the other hand $\iota_M \circ \sigma(\chi)=\sigma (\eta)(\iota_M \circ \sigma(\chi^u))$ where $\sigma(\eta)$ is the (unramified) character of $W_F$ into the connected centre of ${}^LM$ corresponding to $\eta$ via the LLC for (split) tori. Then  $r'(\iota_M \circ \sigma(\chi))$ is the twist by the unramified character $r'(\sigma(\eta))$ of $r'(\iota_M\circ \sigma(\chi^u))$. By section \ref{LS:lc:gamma}, we then get $L(s, r'(\sigma(\eta)(\iota_M \circ \sigma(\chi^u))=L(s+\langle\delta_{r'},\eta\rangle, r'(\iota_M \circ \sigma(\chi^u))$ which is what we want.

The final third stage is to remark that $L(s, \pi(\chi), r)$ is the product of $L(s, \eta\pi_M(\chi^u), r_{M,i})$ where the restriction of $r$ to ${}^LM$ is written as 
a direct sum of irreducible representations $r_{M,i}$, and similarly
$L(s, r \circ \iota  \circ \sigma(\chi))$ is the product of the $L(s, r_{M,i} \circ \iota_M \circ \sigma(\chi))$. That implies the proposition.
\end{proof}

\section{Unramified unitary dual}\label{s:nr:non-uni}

Throughout this section we focus on split classical groups. In this situation, we specify for each classical group in Section \ref{classical:defs}, a general linear group ${\rm GL}_m$ and a classical group ${\bf G}$ of rank $n$ as a maximal Levi subgroup ${\bf M} = {\rm GL}_m \times \bf G$ of a maximal parabolic subgroup $\bf P$ inside an ambient classical group ${\bf H}$ of the same kind and of rank $m+n$.

In this section we prove Theorem \ref{thm:non-uni} in characteristic zero, which we write below as Theorem \ref{thm:zero:non-uni}. Sections \ref{classical:defs} -- 
\ref{sec:unreven} are written for any non-Archimedean local field $F$. In Section \ref{s:Kazhdan}, we transfer the characteristic 0 result to characteristic $p$.

\subsection{Unitary unramified dual} We call unramified a smooth irreducible representation of ${\bf G}(F)$ with non-zero vectors fixed by ${\bf G}(\mathcal O_F)$, and similarly for ${\bf H}(F)$ and ${\rm GL}_{m}(F)$.
 
\begin{theorem}\label{thm:zero:non-uni}
Let $F$ be of characteristic zero. Let $\tau$ be a tempered smooth irreducible unramified representation of ${\rm GL}_m(F)$ and $\pi$ a unitary smooth irreducible unramified representation of ${\bf G}(F)$. Let s be a complex parameter with $\operatorname{Re}(s)>1$. Then the unramified component of
\begin{equation}\label{eq:ind:classical}
I(s,\tau \otimes \pi) = \operatorname{Ind}_{P}^{H}(|\det|^s\tau \otimes \pi)
\end{equation}
is not unitary.
\end{theorem}
 
\begin{remark}
 The theorem remains true if we change our choice of  hyperspecial maximal compact subgroup.
\end{remark}

The proof relies on known classifications, for general linear groups and the split classical groups: that of smooth irreducible unramified representations and that of unitary unramified ones. For the classical groups, the first classification is established only when the characteristic of $F$ is not $2$, and the second one only for $F$ of characteristic $0$.

\subsection{Split classical groups}\label{classical:defs}

In this section, we specify the groups appearing in Theorem \ref{thm:zero:non-uni}, making the notation more precise in order to handle the subtle differences between the classical groups. For the group $\bf G$ we write ${\bf G}_n$ if we wish to specify that it is of rank $n$, namely ${\bf G}_n$ is either ${\rm SO}_{2n}$, ${\rm SO}_{2n+1}$ or ${\rm Sp}_{2n}$. Similarly, for the ambient group $\bf H$ we take ${\bf G}_{n+m}$ of rank $n+m$, a split classical group of the same kind as ${\bf G}_n$.

Let $n$ be a positive integer. We consider on the $F$-vector space $V=F^{n} \oplus F^{n}$ the quadratic form given by
\begin{equation} Q(x_1,\dots ,x_{2n}) =x_1x_{2n}+x_2x_{2n-1}+\cdots+x_{n-1}x_{n+2}+x_{n}x_{n+1}, \end{equation}
where the $x_i$’s are in $F$, and the symplectic form $\omega$ associated to the following matrix in the canonical basis:
\[ \begin{pmatrix} \phantom{-}0_n & J_n \\
-J_n& 0_n\end{pmatrix}, \]
where the $0_n$ is the zero $n\times n$ matrix and $J_n=(\delta_{i,n+1-j})$ is the $n\times n$ matrix with $1$'s along the anti-diagonal and $0$'s elsewhere. We denote by ${\rm SO}_{2n}$ the connected component of the group of  automorphisms of $Q$ and by ${\rm Sp}_{2n}$ the group of automorphisms of $\omega$.

We also consider on the $F$-vector space $V=F^{n}\oplus F \oplus F^{n}$ the quadratic form given by
\begin{equation} Q(x_1,\dots x_n,x,x_{n+1},\dots ,x_{2n}) =x_1x_{2n}+x_2x_{2n-1}+\cdots+x_{n}x_{n+1}+x^2, \end{equation}
where the $x_i$’s and $x$ are in $F$. In this case, we denote by ${\rm SO}_{2n+1}$ the connected component of the group of  automorphisms of $Q$.

Given a split classical group ${\bf G}_n$,  we let ${\bf T}$ be the maximal torus of ${\bf G}_n$ preserving each line $Fe_i$ (where $(e_1,\dots,e_{n})$ is the canonical basis of the first copy of $F^{n}$ in $V$, and $(e_{n+1},\dots,e_{2n})$ that of the second copy). We let ${\bf  B}$ be the Borel subgroup of ${\bf G}_n$ preserving the flag of subspaces $V_i$ of $V$, where $V_i$ for $i=1,\dots,n$ has the basis $(e_1,\dots, e_i)$; we write ${\bf  U}$ for its unipotent radical.

 Furthermore, given another positive integer $m$, we consider the quadratic form $Q'$ on $W=F^m\oplus V \oplus F^m$ given by
 \begin{equation}
 Q'(y_1,\dots,y_m, v, y_{m+1},\dots,y_{2m})=Q(v)+y_1y_{2m}+\cdots+y_my_{m+1},
\end{equation}
and the symplectic form $\omega'$ associated to the matrix \[ \begin{pmatrix} \phantom{-}0_{m+n} & J_{m+n} \\
-J_{m+n}& 0_{m+n}\end{pmatrix}. \]
The connected component of the automorphism group of $Q'$ is ${\rm SO}_{2m+2n}$  in the case of ${\rm SO}_{2n}$, ${\rm SO}_{2m+2n+1}$ in the case of ${\rm SO}_{2n+1}$, and the automorphism group of $\omega'$ is ${\rm Sp}_{2m+2n}$. Using this identification, we consider the parabolic subgroup ${\bf P}$ of ${\bf G}_{m+n}$ stabilizing the first copy of $F^m$ in $W$; it contains the Borel subgroup ${\bf B}$, constructed as above, and it has the Levi subgroup $\bf M$ stabilizing all three components, which is isomorphic to ${\rm GL}_m \times {\bf G}_{n}$, with ${\rm GL}_m$ acting on the first component and ${\bf G}_{n}$ on $V$.

\subsection{Satake paremetrization} Let us first recall the usual parametrization of smooth irreducible unramified representations of ${\bf G}_{n}(F)$ via unramified characters of ${\bf T}(F)$. That parametrization holds for unramified groups over $F$ of any characteristic.
If $\xi$ is an unramified character of ${\bf T}(F)$, then $\operatorname{Ind}_{B}^{G_{n}} (\xi)$ has a unique unramified irreducible subquotient \[\pi(\xi).\] We will call it the unramified component of $\operatorname{Ind}_{B}^{G_{n}} (\xi)$.  All irreducible unramified smooth representations of ${\bf G}_{n}(F)$ are obtained in this way, with a character $\xi$ which is unique up to the action of the Weyl group of ${\bf T}$ in ${\bf G}_{n}$. Concretely 
\begin{equation*}
    \xi(x_1,\ldots,x_n) = \prod_{i=1}^n z_i^{{\rm ord}_F(x_i)},
\end{equation*}
for $(x_1,\ldots,x_n) \in {\bf T}(F)$ and for some $z_i\in (\mathbb C^\times)^n$.  
In this  case we write $\pi(\xi)$ by \[\pi(z_1,\dots,z_n).\]
Such a description applies to ${\rm GL}_m(F)$, using unramified characters of its diagonal torus $\bf A$, and to ${\bf G}_{n+m}(F)$, using unramified characters of its maximal split torus ${\bf A}\times {\bf T}$. An unramified character of ${\bf A}(F)$ is given by an $m$-tuple of non-zero complex numbers $(y_1,\dots, y_m)$, and we write $\tau(y_1,\dots,y_m)$ for the corresponding unramified irreducible component; 
similarly an unramified character of ${\bf G}_{m+n}(F)$ is given by an $(m+n)$-tuple of non-zero complex numbers $(y_1,\dots, y_m, z_1,\dots, z_n)$, and gives 
the unramified irreducible component 
\[\pi(y_1,\dots,y_m,z_1,\dots, z_n).\]
Since the representation $\tau$ of ${\rm GL}_m(F)$ is supposed to be tempered, it has the form $\tau(y_1,\dots, y_m)$ where the $y_j$ have modulus $1$, and $\tau$ is the full parabolically induced representation.

The representation $\pi$ is supposed to be unitary. To interpret that condition concretely, in terms of the parameters with $\pi=\pi(z_1,\dots,z_n)$, we now assume $F$ of characteristic $0$, to be able to use the results of Mui\'c \cite{Mu2006} and Mui\'c-Tadi\'c \cite{MuTa2011}. Note that in those papers the notion of unramified representation refers to the same choice of hyperspecial maximal compact subgroup as ours,
viz. the group ${\bf G}_{n}(\mathcal O_F)$.

In \cite{Mu2006}  Mui\'c gives a finer description of unramified smooth irreducible representations of ${\bf G}_{2n}(F)$, which obviously also applies to ${\bf G}_{2m+2n}(F)$, in three stages: strongly negative representations, negative representations (which are unitary), general case. The classification of unitary unramified representations in \cite{MuTa2011} uses the classification of \cite{Mu2006}. However we have to be careful in the case where ${\bf G}_n={\rm SO}_{2n}$, indeed  both references consider the group ${\rm O}_{2n}$ instead of our group ${\rm SO}_{2n}$, and an irreducible smooth representation of ${\rm O}_{2n}(F)$ is called unramified if it has non-zero fixed vectors under ${\rm O}_{2n}(\mathcal O_F)$.

\subsection{Unramified representations of even orthogonal groups}\label{sec:unreven}
The major difference with ${\rm SO}_{2n}$ occurs already when $n=1$, and concerns the reducibility of (unramified) principal series. Indeed ${\rm SO}_2(F)$ is a split torus, and a principal series is simply a character $\xi$, in particular irreducible; but ${\rm  O}_2(F)-{\rm SO}_2(F)$ acts on ${\rm SO}_2(F)$ by inversion, so $\xi$ induces irreducibly to ${\rm O}_2(F)$, (to a unitary representation if and only if $\xi$ is unitary), unless its square is trivial, in which case $\xi$ induces to the direct sum of its two extensions to ${\rm  O}_2(F)$, which are both unitary and can be distinguished by their value on the transposition matrix, which is a sign. But ${\rm  O}_2(\mathcal O_F)$ contains that transposition matrix, so if $\xi$ is unramified there is indeed a unique unramified component in the induced representation, where the transposition matrix acts trivially.

That phenomenon persists for all positive integers $n$. Indeed the normalizer of ${\bf  T}$ in ${\rm O}_{2n}$ is twice bigger than the normalizer in ${\rm SO}_{2n}$; for example it contains the transposition matrix $\sigma$ which exchanges $e_n$ and $e_{n+1}$, and acts on an unramified character $\xi$ of ${\bf  T}(F)$ by changing $z_n$ to its inverse in the parameter of $\xi$, yielding a character $\xi^\sigma$. We can inflate $\xi$ to ${\bf B}(F)$, induce first to ${\rm SO}_{2n}(F)$, where, as we discussed previously, the induced representation $I(\xi)$ has a line of ${\rm SO}_{2n}(\mathcal O_F)$-fixed vectors and a unique unramified irreducible component $\pi(\xi)$. We further induce to a representation of ${\rm  O}_{2n}( F)$, that we denote by  \[I^+(\xi).\] It is clear that the restriction of $I^+(\xi)$ to ${\rm SO}_{2n}(F)$ is the direct sum of $I(\xi)$ and $I(\xi^\sigma)$, with the two lines of ${\rm SO}_{2n}(\mathcal O_F)$-fixed vectors, exchanged by $\sigma$. But ${\rm O}_{2n}(\mathcal O_F)$ is generated by ${\rm SO}_{2n}(\mathcal O_F)$ and $\sigma$, hence $I^+(\xi)$ has a unique line of ${\rm  O}_{2n}(\mathcal O_F)$-fixed vectors. Thus,  $I^+(\xi)$ has a unique unramified irreducible component, that we denote by \begin{equation*}
   \pi^+(\xi). 
   \end{equation*}
We note that the direct sum of $\pi(\xi)$ and $\pi(\xi^\sigma)$ occurs as a subquotient of $I^+(\xi)$, with the two factors exchanged by $\sigma$.   If those two factors are not isomorphic then the direct sum is an irreducible component of $I^+(\xi)$ and therefore, it corresponds to the unique unramified irreducible component $\pi^+(\xi)$. If they are isomorphic, then $\pi(\xi)$ extends to ${\rm O}_{2n}(F)$; there are two such extensions, one being the twist of the other by the nontrivial character of ${\rm  O}_{2n}(F)$ trivial on ${\rm SO}_{2n}(F)$, but only one of them, is the unramified irreducible component of $I^+(\xi)$. 

Note that if $\eta$ is an unramifed character of ${\bf  T}(F)$ then $\pi^+(\xi)=\pi^+(\eta)$ if and only if $\pi(\eta)$ is equal to $\pi(\xi)$ or $\pi(\xi^\sigma)$. It follows that $\xi$ goes to $\pi^+(\xi)$ gives all unramified irreducible smooth representations of ${\rm  O}_{2n}(F)$, and that $\pi^+(\xi)$ determines $\xi$ up to the action of the normalizer of ${\bf  T}(F)$ in ${\rm  O}_{2n}(F)$, which acts on the parameters $(z_1,\dots,z_n)$ by permutation of the indices and sending some of the $z_i$’s to their inverses. Therefore, we deduce the following.

\begin{proposition}\label{prop:unrunitary}
The representation $\pi^+(\xi)$ can be unitary only if $\pi(\xi)$ is; if that is the case, then $\pi(\xi^
\sigma)$ is also unitary, and so is $\pi^+(\xi)$.
\end{proposition}
Proposition \ref{prop:unrunitary}  means that a classification of unitary smooth irreducible unramified representations of ${\rm  O}_{2n}( F)$ can be directly applied to ${\rm SO}_{2n}(F)$ instead. Moreover, a criterion of irreducibility of $\pi(\xi)$ in terms of $\xi$ has to be unsensitive to replacing $\xi$ with $\xi^\sigma$.

\subsection{Unitary dual in characteristic zero} We write ${\bf G}_n^+$ to denote ${\rm O}_{2n}$ in the special case of ${\bf G}_n = {\rm SO}_{2n}$, and ${\bf G}^+_n={\bf G}_n$ in the two other cases.

We now recall what we need of the classifications of \cite{Mu2006}  and \cite{MuTa2011}, following the introduction of \cite{MuTa2011}. The classification of general smooth irreducible unramified representations of ${\bf G}_n^+(F)$ in terms of negative ones will be enough for us.

By \cite{Mu2006}  (see \cite{MuTa2011} definition 0-6 and surrounding comments) that general case is as follows: one considers a multiset $E$ of triples $(r, \xi, \alpha)$ where $r$ is a positive integer, $\xi$ a unitary unramified character of $F^\times$ and $\alpha$ a positive real number, and a negative representation $\rho$ of ${\bf G}^+_{l}(F)$. Then one attaches to those data the unique irreducible unramified component $\pi(E, \rho)$ of the representation of ${\bf G}^+_{n}(F)$ parabolically induced from tensor (over triples) $(\nu^\alpha \xi)(r)\otimes \rho$, where $\eta(r)$ for a character $\eta$ of $F^\times$ is the character $\eta \circ \det$ of ${\rm GL}_r(F)$. In that manner we get all irreducible unramified representations $\pi^+$ of ${\bf G}^+_{n}(F)$, up to isomorphism (the way we order the triples to construct the tensor product does not matter), and the representation $\pi^+$ determines $\rho$ (up to isomorphism) and the multiset of triples $E$.  

Now if $\pi(E, \rho)$ is unitary, then (\cite{MuTa2011}, second assertion of Theorem 0-8 and definition 0-7)
\begin{equation}\label{def:triplesunitary}
    \text{ for all triples  $(r, \xi, \alpha)$ in $E$ we have $\alpha<1$}
\end{equation} 
(\cite{MuTa2011}  gives necessary and sufficient conditions for $\pi(E, \rho)$ to be unitary, but we do not need them). Now we are in position to give the proof of the proposition when the characteristic of $F$ is zero.

\begin{proof}[Proof of Theorem \ref{thm:zero:non-uni}] First, we recall the notation of the theorem: ${\bf G}={\bf G}_n$,  ${\bf H}={\bf G}_{n+m}$ and ${\bf M}={\rm GL}_m\times{\bf G}_n$ (See Section \ref{classical:defs}). Moreover, we set ${\bf G}^+$ to be ${\bf G}^+_n$ and ${\bf H}^+$ to be ${\bf G}^+_{n+m}$.  Our representation $\pi$ of ${\bf G}(F)$ has the form $\pi(\xi)$ for some unramified character $\xi$ of ${\bf  T}(F)$. We let $\pi^+(\xi)$ be the representation constructed in Section \ref{sec:unreven}, above, in the case ${\bf G}={\rm SO}_{2n}$ and simply $\pi(\xi)$ in the other cases. The representation $\pi^+(\xi)$ of ${\bf  G}^+(F)$ (which is unramified unitary) has the preceding form $\pi(E, \rho)$. We also have that the unramified tempered representation $\tau$ of ${\rm GL}_m(F)$ given by $m$ unramified unitary characters.

We want to consider the unramified irreducible component $\pi'$ of the representation of ${\bf H}^+(F)$ parabolically induced by $\nu^s\tau \otimes  \pi^+(\xi)$, where $t=\operatorname{Re}(s)>1$. We see $\nu^s\tau$ as given by $\nu^t\eta_1,\dots,\nu^t\eta_m$ for unitary unramified characters $\eta_j$ of $F^\times$. Then $\pi'$ is simply $\pi(E', \rho)$ where $E'$ is obtained from $E$ by adding the triples $(1,\eta_i, t)$ for $i=1$ to $m$. Indeed using parabolic induction in stages, we see that $\pi(E', \rho)$ is the unramified component of the induction of $\nu^s\tau\otimes \pi(E,\rho)$. It follows  that $\pi'$ does not satisfy the condition \eqref{def:triplesunitary} above, so cannot be unitary. Furthermore, since in the case where ${\bf G}^+ ={\rm O}_{2n}$,  the components of the restriction of $\pi'$ to ${\rm SO}_{2n}(F)$ cannot be unitary either by Proposition \ref{prop:unrunitary}, thus we obtain that the unramified component of $\operatorname{Ind}_{P}^{G_{m+n}}(|\det|^s\tau \otimes \pi)$ is not unitary.
\end{proof}

\section{\`A la Kazhdan}\label{s:Kazhdan}

 In this final section we assume that $\bf G$ is split, and use the method of close local fields \`a la Kazhdan to deduce Theorem \ref{thm:non-uni}. In characteristic 0, it is Theorem~\ref{thm:zero:non-uni} of the previous section. The principle is well-known; essentially, all we need has been established by Ganapathy in \cite{Ga2015} and used by her to prove the local Langlands correspondence for ${\rm GSp}_4(F)$ when $F$ has positive characteristic. We thus proceed with the same spirit of the Kahzdan philosophy in Section  \ref{ss:proof:p:non-uni}, which permits the transfer of Theorem 
 \ref{thm:zero:non-uni} to positive characteristic, hence concluding the proof of Theorem \ref{thm:non-uni}.

\subsection{} \label{subsec:Kazhdan}
Assume that $\bf G$ is split. It then comes from a Chevalley group scheme over $\mathbb{Z}$, which we also write $\bf G$. We choose ${\bf T} = {\bf S}$, $\bf B$, $\bf U$ to be group schemes over $\mathbb{Z}$ as well.

The subgroup $K = {\bf G}(\mathcal{O}_F)$ is a hyperspecial maximal compact subgroup of $G = {\bf G}(F)$, and the quotient of $K$ by its pro-$p$ radical $K_1$ identifies with ${\bf G}(\kappa_F)$. The inverse image in $K$ of ${\bf B}(\kappa_F)$ is an Iwahori subgroup $I$ of $G = {\bf G}(F)$. The profinite group $I$ has a filtration by open normal pro-$p$ subgroups $I_m$, for positive integers $m$, and $I_1$, the inverse image in $K$ of ${\bf U}(\kappa_F)$, is the pro-$p$ radical of $I$. We write $\mathcal{H}_m$ for the Hecke algebra of $I_m$ in $G$.

Let ${\rm Rep}(G)$ be the category of smooth (complex) representations of $G$ and ${\rm Rep}_m(G)$ the subcategory of representations generated by their vectors fixed under $I_m$; ${\rm Irrep}(G)$ and ${\rm Irrep}_m(G)$ denote the irreducible representations, respectively. By Proposition 3.15 of \cite{Ga2015}, ${\rm Rep}_m(G)$ is a direct factor of ${\rm Rep}(G)$ and the functor
\[\pi\mapsto \pi^{I_m}\]
of fixed points under $I_m$ yields an equivalence of categories between ${\rm Rep}_m(G)$ and the category of modules over the Hecke algebra $\mathcal{H}_m$. Moreover, by a result of Ciubotaru, Corollary 1.3 of \cite{Ci2018}, an irreducible smooth representation in ${\rm Rep}_m(G)$ is unitary if and only if the corresponding module over $\mathcal{H}_m$ (which is irreducible) is unitary.

Let $F'$ be another locally compact non-Archimedean field, with residue characteristic $p$. Given a positive integer $m$, assume that $F$ and $F'$ are $m$-close in the sense of Kazhdan. More precisely, one is given a ring isomorphism $\Lambda$ of $\mathcal{O}_F / \mathfrak{p}_F^m$ onto $\mathcal{O}_{F'} / \mathfrak{p}_{F'}^m$. Now, for the ease of writing, set $G' = {\bf G}(F')$ and denote the Hecke algebra over $F'$ by $\mathcal{H}'$; and similarly for $I'$, $I_m'$, etc. Furthermore, choose a uniformizer $\varpi$ of $F$ and a uniformizer $\varpi'$ of $F'$ such that $\Lambda(\varpi) \equiv \varpi' \pmod{ \mathfrak p_{F'}^n}$. Then, in Theorem 3.13 of \cite{Ga2015}, Ganapathy derives from $\Lambda$, and the choices of $\varpi$ and $\varpi'$, a ring isomorphism $\zeta_m$ of $\mathcal{H}_m$ onto $\mathcal{H}_m'$.

Now, from the ring isomorphism $\zeta_m$ there follows an equivalence of categories of ${\rm Rep}_m(G)$ onto ${\rm Rep}_m(G')$. If $\pi \in {\rm Rep}_m(G)$, $\pi' \in {\rm Rep}_m(G')$, we say that $\pi$ and $\pi'$ correspond if the associated Hecke algebra modules are isomorphic via $\zeta_m$ i.e. we have an isomorphism of vector spaces $\kappa_m$, such that  for every $f\in \mathcal{H}_m$, we have the following commuting diagram. 
\[\begin{tikzcd} \pi^{I_m} \arrow[r,"f"] \arrow[d,"\kappa_m"'] & \pi^{I_m} \arrow[d,"\kappa_m"]\\
\pi'^{I'_m}\arrow[r,"\zeta_m(f)"']& \pi'^{I'_m}
\end{tikzcd}\]
 Then, up to isomorphism, $\pi'$ is determined by $\pi$ (and conversely).

Let $\pi \in {\rm Irrep}(G)$, with corresponding $\pi' \in {\rm Irrep}(G')$. Looking at growth conditions on coefficients, one easily proves that $\pi$ is cuspidal (resp. square integrable) if and only if $\pi'$ is, Theorem 4.6 of \cite{Ga2015}. Moreover, if they are square integrable, then for Haar measures giving volume 1 to Iwahori subgroups, the formal degrees (which are real numbers) for $\pi$ and $\pi'$ are the same, Corollary 4.5 of [\emph{loc.\,cit.}]. Now, the result mentioned above by Ciubotaru \cite{Ci2018}, is used in the proof of Lemma~5.2 of \cite{Lom2019} in order to obtain that
\begin{equation}\label{eq:Kazuni}
    \text{$\pi$ is unitary if and only if $\pi'$ is.}
\end{equation}

The algebra $\mathcal{H}_m$ contains as a subalgebra the group ring of $I/I_m$, and there is a group isomorphism of $I/I_m$ onto $I'/I_m'$. By restriction, the isomorphism $\zeta_m$ gives an isomorphism of that group ring with the group ring of $I'/I_m'$. For a positive integer $l < m$, the characteristic funcion of $I_l/I_m$ is sent to the characteristic function of $I_l'/I_m'$, and it follows that $\pi$ belongs to ${\rm Rep}_l(G)$ if and only if $\pi'$ belongs to ${\rm Rep}_l(G')$. Also, using property (a) of Theorem 3.13 of \cite{Ga2015} and the notation there applied to $s_i \in S_1$, results in the characteristic function of $K/I_m$ being sent to the characteristic function of $K'/I_m'$. Therefore
\begin{equation} \label{eq:Kazunr}
    \text{$\pi$ is unramified if and only if $\pi'$ is.}
\end{equation}

Let us consider parabolic induction. Following \cite[Section 4.3]{Ga2015}, we assume that $F$ and $F'$ are $4$-close, and choose a ring isomorphism $\Lambda$ of $\mathcal O_F/\mathfrak{p}^4_F$ onto $\mathcal O_{F'}/\mathfrak{p}'^4_{F'}$, and uniformizers $\varpi$, $\varpi'$ of $F$, $F'$ compatible with $\Lambda$. Let ${\bf P}$ be a parabolic subgroup of ${\bf G}$ containing ${\bf B}$, and ${\bf M}$ its Levi subgroup containing ${\bf T}$, ${\bf N}$ its unipotent radical (provisionally those letters do not stand for integers). Then $I_{M}=I \cap {\bf M}(F)$ is an Iwahori subgroup of ${\bf M}(F)$, with Hecke algebra $\mathcal H_M$ and $K_M=K \cap {\bf M}(F)$ a hyperspecial  maximal compact subgroup of ${\bf M}(F)$. We put primes for the corresponding objects over $F'$.
Let $\tau$ be a smooth irreducible  representation representation of ${\bf M}(F)$ with $I_M$ fixed vectors, and $\tau'$ a  representation of ${\bf M}(F')$ that corresponds to $\tau$. Then $\operatorname{Ind}_{{\bf P}(F)}^{{\bf  G}(F)} \tau$ corresponds to $ \operatorname{Ind}_{{\bf  P}(F')}^{{\bf  G}(F')} \tau'$ (\cite[Lemma 4.10 \& Theorem 4.14]{Ga2015}). We remark that we are applying Theorem 4.14 of \cite{Ga2015} for $m=1$, and that is why we use $(m+3)=4$ close fields.

\subsection{Transfer  à la Kazhdan}\label{ss:proof:p:non-uni}
We transfer the characteristic 0 proof of Theorem \ref{thm:zero:non-uni} of the previous section to the positive characteristic case. We operate that transfer using close local fields à la Kazhdan as in Section \ref{subsec:Kazhdan}.

\begin{proof}[Proof of the Theorem \ref{thm:non-uni} in characteristic $p$]
We now choose $F'$ of characteristic $0$ such that $F$ and $F'$ are $4$-close (that is possible), and apply the considerations of Section \ref{subsec:Kazhdan} to the ambient group ${\bf H}$ with the Levi subgroup ${\bf M}={\rm GL}_m\times{\bf G}$. We have the representation $\tau$ of ${\rm GL}_m(F)$ and the representation $\pi$ of ${\bf G}(F)$, and corresponding representations $\tau'$ and $\pi'$ obtained via the previous process. 

We also have the complex number $s$ with $\operatorname{Re}(s)>1$ and our goal is to show that the unramified irreducible component of the representation of ${\bf H}(F)$ parabolically induced from  $\nu^s\tau\otimes \pi$ is not unitary. By Theorem \ref{thm:zero:non-uni}, the result is true over the characteristic $0$ field $F'$, and it is enough to show that the hypotheses on $\tau$ and $\pi$ transfer to the corresponding hypotheses on $\tau'$ and $\pi'$, and that the result over $F'$ transfers back to $F$. First, we transfer the hypotheses.

The representation $\tau$ is tempered, parabolically induced from the unitary character $\xi= (\xi_1,\dots,\xi_m)$ of ${\bf A}(F)$, and the compatibility with parabolic induction recalled in Section \ref{subsec:Kazhdan} shows that $\tau'$ is induced from the unitary character $(\xi'_1,\dots, \xi'_m)$ with $\xi'_i$ taking the same value as $\xi$ on uniformizers. Moreover, twisting $\tau$ by $\nu^s$ corresponds to twisting $\tau'$ by $\nu'^s$, since twisting by $\nu^s$ amounts to multiplying all $\xi_i$'s by the character $\nu^s$ of $F^\times$. Since the representation $\pi$ is smooth irreducible unitary unramified, properties \eqref{eq:Kazuni} and \eqref{eq:Kazunr} imply that $\pi'$ is also smooth 
irreducible unitary unramified. Furthermore, the representation $I(s,\tau\otimes \pi)$  parabolically induced from $\nu^s\tau \otimes \pi$ corresponds to the representation $I(s,\tau'\otimes \pi')$ parabolically induced from $\nu'^s\tau' \otimes \pi'$. Finally, using again property \eqref{eq:Kazunr}, we have that the unramified irreducible component of $I(s,\tau\otimes \pi)$ corresponds to the unramified irreducible component of $I(s,\tau'\otimes \pi')$.

To finish, we need to transfer back the result over $F'$ to $F$. Indeed, Theorem \ref{thm:zero:non-uni} holds over $F'$, hence the unique unramified irreducible subquotient of $I(s,\tau'\otimes \pi')$ is not unitary. Therefore, property \eqref{eq:Kazuni} implies that the unique unramified irreducible subquotient of $I(s,\tau\otimes \pi)$ is not unitary either, which is what we wanted. \end{proof}

\subsection{Remarks}

It is worth noting that Theorems \ref{TLC} and \ref{SMC} in positive characteristic from the case of characteristic $0$ can be obtained à la Kazhdan, proceeding in the same spirit of the proof of Theorem \ref{thm:non-uni}. Of course, split groups are included in the results of the previous sections, and Theorem~\ref{TLC} in particular is in fact proved this way in Corollary 5.5 of \cite{Lom2019}. However, following the alternate treatment of the two theorems may lead to new insights. The hope is that the comparison of Hecke algebras used above can be extended to the general quasi-split case, where partial progress is being made \cite{Ga2019, Ga2021}.

\appendix

\section{On ${\rm SL}_2$ and ${\rm SU}_{2,1}$}\label{SL:SU}
\setcounter{section}{1}

We study the structure of  unramified principal series representations of ${\rm SL}_2$ and ${\rm SU}_{2,1}$, and the vectors that are fixed by open compact subgroups. In order to do that, we study the Iwahori-fixed vector of these representations, the Iwahori Hecke algebras and the equivalence between Iwahori blocks  and the categories of Iwahori Hecke modules \cite{B1976}.

\subsection{${\rm SL}_2$}
Let ${\bf G} = {\rm SL}_2$ and $G = {\rm SL}_2(F)$. We setup basic objects and notions. Let $\bf B$ the Borel subgroup of ${\bf G}$ consisting of upper triangular matrices and $\bf T$ the subgroup of ${\bf G}$ consisting of diagonal matrices.  In this setup, the modulus character is\[\delta\begin{pmatrix}t & 0\\
0 & t^{-1}\end{pmatrix}=|t|^2_F,\]
where $t\in F^\times$.

 Let $K={\rm SL}_2(\mathcal O
_F)$ and $I$ the Iwahori subgroup contained in $K$ consisting of upper triangular matrices modulo $\mathfrak p_F$. We note that 
\begin{equation}
   K=I\cup IwI,\text{ where } w=
   \begin{pmatrix}
      \textcolor{white}{e,}0&1\ \\
      -1&0 \ 
\end{pmatrix}.
\end{equation}
We also consider the following maximal compact open subgroup of $G$
\[K'=I\cup IsI,\text{ where } s=
\begin{pmatrix}
   0&\varpi \ \\
   -\varpi^{-1}&0 \ 
\end{pmatrix},\]
where $\varpi$ is a uniformizer of $F$.
Finally, let $T_w$ and $T_s$ be the  characteristic functions of $IwI$ and of $IsI$, respectively. In the Iwahori Hecke algebra, we have
\[ (T_w+1)(T_w-q)=0 \quad \text{ and } \quad (T_s+1)(T_s-q)=0. \]

\subsubsection{Iwahori Hecke algebras and eigenspaces}

For any unramified character $\chi:T\to \mathbb \mathbb \mathbb C^\times$, the dimension of $W=\operatorname{Ind}_B^{G}(\chi)^I$ is equal to $2$. The smooth functions from $G$ to $\mathbb C$ characterized by 
\[e(g)=\begin{cases}1& g\in I \\ 0& g \not \in BI\end{cases} \quad \text{ and } \quad  f(g)=\begin{cases}1& g\in wI \\ 0& g \not \in BwI\end{cases} \]
form a basis of $W$. The function $T_w$ and $T_s$ induced two linear maps on $W$.  First, we  focus on $T_w$. This linear map operates on the basis vectors $e,f$ as follows
\[T_w(e)=f \quad \text{ and } \quad  T_w(f)=qe+(q-1)f.\]
From this formula, we have the following eigenspace decomposition: the eigenvalues of $T_w$ are $-1$ and $q$, with associated space given by
\begin{equation}\label{eq:eigenSL2w}
W_{w,-1}=\mathbb C(qe-f) \quad \text{ and } \quad W_{w,q}=\mathbb C(e+f),    
\end{equation}
respectively.

Now we concentrate on $T_s$. In this case, we have
\[T_s(e)=(q-1)e+(q/a)f \quad \text{ and } \quad T_s(f)=ae,\]
where $a=\chi \delta^{1/2}(\operatorname{diag}(\varpi^{-1}_F, \varpi_F))$. Indeed, let us first compute $T_s(f)$. By definition, we have that
\[T_s(f)=\sum_{u(I\cap sIs^{-1})\in I/(I\cap sIs^{-1})}us\cdot f.\]
Thus, we need to identify  $I\cap sIs^{-1}$.
Using the Iwahori decomposition we have that 
\[I={\bf U}(\mathcal O_F){\bf T}(\mathcal O_F)\overline {\bf U}(\mathfrak p_F)\] and that
\[sIs^{-1}= w{\bf U}(\mathfrak p_F^2){\bf T}(\mathcal O_F)\overline {\bf U}(\mathfrak p_F^{-1})w^{-1}={\bf U}(\mathfrak p_F^{-1}){\bf T}(\mathcal O_F)\overline {\bf U}(\mathfrak p_F^2),\] so $I\cap sIs^{-1}={\bf U}(\mathcal O_F){\bf T}(\mathcal O_F)\overline {\bf U}(\mathfrak p_F^2)$. This identification allows us to check that inclusion of $\overline {\bf U}(\mathfrak p_F)$ into $I$ gives a bijection of $\overline {\bf U}(\mathfrak p_F)/\overline {\bf U}(\mathfrak p_F^2)$ onto $I/(I \cap sIs^{-1})$. So, we have that \[T_s(f)(1)= \sum_{u\overline {\bf U}(\mathfrak p_F^2)\in \overline {\bf U}(\mathfrak p _F)/\overline {\bf U}(\mathfrak p_F^2)}f(us).\] 
But if $u$ has bottom row $(\varpi x,1)$, $us$ has bottom row $(\varpi, -x)$, so $us\in BI$ unless $x$ is in $\mathfrak p_F$. For $x\in \mathfrak p_F$, we may as well take $x=0$, that is $u=1$. Thus $T_s(f)(1)=a$ where $a=\chi \delta^{1/2}(\operatorname{diag}(\varpi^{-1}, \varpi))$. On the other hand, we also consider
\[T_s(f)(w)= \sum_{[u]\in \overline U(\mathfrak p _F)/\overline U(\mathfrak p_F^2)} f(wus).\] Since $wus=wuw^{-1}ws\in B$, $T_s(f)(w)=0$. Therefore $T_s(f)=ae$. 
It follows that $T_s(e)=(q-1)e+(q/a)f$, and $T_s(ae+f)=q(ae+f)$ and $T_s(ae-qf)=-(ae-qf)$. 

Now, from these formulas that we just checked, we have the following eigenspace decomposition: the eigenvalues of $T_s$ are $-1$ and $q$, with associated space given by
\begin{equation}\label{eq:eigenSL2s}
W_{s,-1}=\mathbb C(ae-qf) \quad \text{ and } \quad W_{s,q}=\mathbb C(ae+f),    
\end{equation}
respectively.

Thanks to the equivalence between Iwahori blocks  and the categories of Iwahori Hecke modules \cite{B1976}, there is reducibility if and only if $T_w$ and $T_s$ have a common eigenspace. Thus if we combine equations \eqref{eq:eigenSL2w} and \eqref{eq:eigenSL2s}, we check that this happens in three cases:
\begin{enumerate}
    \item $a=1$, with the trivial representation as subrepresentation. Indeed, in this case the subspace $\mathbb C(e+f)$ invariant under the action of the Iwahori Hecke algebra. As the fixed by $I$ vector of the trivial representation is indeed $\mathbb C(e+f)$ and the action of the Iwahori Hecke algebra is given by multiplication by $q$, we have a non-zero morphism  between representations of $G$, from the trivial representation to $i_B^G(\chi)$. 
    \item $a=-q$, with $e+f$ fixed by $K$ and $qe-f$ fixed by $K'$.
    \item $a=q^2$, with the trivial representation as quotient. Indeed, we note that the corresponding $a$ of the the dual representation is equal to $\chi^{-1}\cdot \rho(\operatorname{diag}(\varpi^{-1}_F, \varpi_F))=q^{-1}\cdot q=1$. Thus, from point (1), the dual representation has the trivial representation as subrepresentaion. From this, we obtain that $i_B^G(\chi)$ has the trivial representation as a quotient.
\end{enumerate}
To finish our discussion in the ${\rm SL}_2$ case, we are going to analyze the case when the representations obtained above are generic. Let us fix a non-trivial character $\psi$ of $F$. Let also $\theta=\theta_\psi$ be the character
of $U$  given by
\[\theta_\psi\begin{pmatrix}
   1& x \\
   0&1
\end{pmatrix}=\psi(x), \text{ for }x\in F. \]
We recall a formula for a Whittaker functional $\lambda$ on ${\rm Ind}_B^G \rho^{-1}\eta$ for $\eta$ an
unramified character of the torus $T$. That formula is given as follows, first we fix a Haar measure $du$ on $U$ so that $U(\mathcal O_F)$ has measure equal to $1$. Then the formula for $\lambda$ is 
\[\lambda(\phi)=\int_U\phi(wu)\overline \theta(u)du,\]
where $\phi$ is a function in $\operatorname{Ind}_B^G\rho^{-1} \eta$.
The integral on the right hand side converges absolutely only for functions with support in $BwU$,
and in general it converges conditionally (or ``in principal part") in the
sense that the integrals
on $U(\mathfrak p_F^r)$ are stationary as $r$ tends to $-\infty$. Indeed the functions in ${\rm Ind}_B^G
\rho^{-1}\eta$ are characterized
by their restriction to $wU$, and those restrictions are, as functions of $x$
where the first row of $u$ is $(1,x)$,
all locally constant functions which have the form a constant times
$a^{v_F(x)}$ when $v_F(x)$ tends to $-\infty$;
if that valuation $v=v_F(x)$ is fixed, integrating against $\overline\theta$ gives $0$ as soon
as $\psi$ is non-trivial on $\mathfrak p_F^{v-1}$.

So let us  compute $\lambda(e)$ and $\lambda(f)$ (in terms of our parameter $a$, which in our case is equal to $\eta(\operatorname{diag}(\varpi^{-1},\varpi))$,
even though only $a=-q$
is of interest to us). First we compute $\lambda(f)$. Since $f$ has support in $BwI=BwU(\mathcal O_F)$, included in $BwU$,
one computes immediately $\lambda(f)=1$, if $\psi$ is trivial on $\mathcal O_F$, $0$ otherwise. Now we compute $\lambda(e)$. In order to do it, one needs to see when $wu$ can be in $BI=B\overline U(\mathfrak p_F)$. But
elements of $BI$ are those
matrices with last row $(c,d)$ satisfying $v_F(c)>v_F(d)$, so we need the first
row $(1,x)$ of $u$ to satisfy
$v=v_F(x)<0$. Moreover $e(wu)$ is then $a^v$. If $\psi$ is non-trivial on $\mathcal O_F$, we find
again $\lambda(e)=0$.
So let us assume $\psi$ is trivial on $\mathcal O_F$, but not on $\mathfrak p_F^{-1}$. Then the only
non-zero contribution
to the integral $\lambda(e)$ comes from $v=-1$, where it is $-1/a$.

Now we use these computation to check generecity of the representation obtained from unramified principal series representations.
\begin{enumerate}
    \item In the case where $a=1$, we find  $\lambda(e+f)=0$, which is natural because
the trivial subrepresentation is not generic.
\item In the case where $a=-q$, we find the following: if  $\psi$ is trivial on $\mathcal O_F$, but not on $\mathfrak p_F^{-1}$. Then $\lambda(e+f)$ not zero, thus the $\theta$-generic
subrepresentation  is the one
with $K$-fixed vectors. To finish, we note that by using the action of $T(F)$ we see that when $\psi$ has even
exponent, i.e. is trivial
on $\mathfrak p_F^k$ but not on $\mathfrak p_F^{k-1}$ for even $k$, then the $\theta$-generic
subrepresentation  is again the one
with $K$-fixed vectors. On the other hand, using the action of the normalizer
of $I$ in ${\rm GL}_2$, we find
that when $\psi$ has odd exponent, the $\theta$-generic subrepresentation is
the one with $K'$-fixed vectors.
\end{enumerate}

\subsection{${\rm SU}_{2,1}$}
Let $E$ be a quadratic unramified extension $F$ and $\varpi_F$ a uniformizer of $F$. By definition, we have that $\varpi_F$ is also a uniformizer of $E$. We consider the following hermitian form $h$ on $V=E^3$
\[z_1\overline w_3+z_2\overline w_2+z_3\overline w_1.\]
We note that its associated matrix in the canonical basis $(e_1,e_2,e_3)$ is equal to 
\[\begin{pmatrix}0&0&1 \\
0&1&0 \\
1&0&0 \end{pmatrix}.\]
We denote by ${\rm SU}_{2,1}$ the group of automorphisms of $h$ with determinant $1$. It is a connected reductive group over $F$. Moreover, since $E$ is unramified over $F$, ${\rm SU}_{2,1}$ is unramified. Finally, let ${\bf G}={\rm SU}_{2,1}$ and $G = {\rm SU}_{2,1}(F)$, its $F$-points.

Now, let $\bf T$ be the maximal torus given by the elements of ${\bf G}$ that preserve each line $Ee_i$ and $\bf S$ the maximal split torus in $\bf T$. The $F$-points of the torus $\bf T$ is identified with  $E^\times$ via the following map
\begin{align*}
    d\colon E^\times &\to G\\
    z &\mapsto d_z=\operatorname{diag}(z, \overline z \cdot z^{-1},\overline z^{-1}).
\end{align*} Moreover, the torus $S$ is identified with $F^\times$ via $d$. Now, we let $\bf B$ be the Borel subgroup containing $\bf T$ consisting of element preserving the flag subspaces $V_i$ of $V$, where $V_i$ for $i=1,2,3$ has the basis $(e_1,e_2,e_3)$; we write $\bf U$ for its unipotent radical.  Finally the modulus character is\[\delta(d_z)=|z|^4_E.\] 

 Let $K={\rm SU}_{2,1}(\mathcal{O}_F)$ be a hyperspecial maximal compact subgroup, i.e., the stabilizer of the lattice $\mathcal O_E^3$ of $V$ and $I$ the Iwahori subgroup contained in $K$ consisting of upper triangular matrices modulo $\mathfrak{p}_E$. We note that 
\[K=I\cup IwI,\text{ where } w=\begin{pmatrix}0&\phantom{-}0&1 \\
0&-1&0 \\
1&\phantom{-}0&0\end{pmatrix}.\]
We also consider the following maximal compact subgroup of $G$
\[K'=I\cup IsI,\text{ where } s=\begin{pmatrix}0&\phantom{-}0&\varpi_F^{-1} \\
0&-1&0^{\phantom{-1}} \\
\varpi_F&\phantom{-}0&0^{\phantom{-1}}\end{pmatrix}.\]
We note that $K'$ is the stabilizer of the lattice ${\frac{1}{\varpi_F}}\mathcal{O}_E\oplus \mathcal{O}_E\oplus \mathcal{O}_E$ of $V$. Thus, it is a special, but not hyperspecial \cite[Section 3.11]{Ti1979}.

Finally, let $T_w$ and $T_s$ be the  characteristic functions of $IwI$ and of $IsI$, respectively. In the Iwahori Hecke algebra we have that
\[(T_w+1)(T_w-q^3)=0 \quad \text{ and } \quad (T_s+1)(T_s-q)=0.\]

\subsubsection{Iwahori Hecke algebras and eigenspaces}

For any unramified character $\chi:T\to \mathbb \mathbb \mathbb C^\times$, the dimension of $W=\operatorname{Ind}_B^{G}(\chi)^I$ is equal to $2$. The smooth functions from $G$ to $\mathbb C$ characterized by 
\[e(g)=\begin{cases}1& g\in I \\ 0& g \not \in BI\end{cases} \quad \text{ and }  \quad  f(g)=\begin{cases}1& g\in wI \\ 0& g \not \in BwI\end{cases} \]
form a basis of $W$. The function $T_w$ and $T_s$ induced two linear maps on $W$. First, we are going to focus on $T_w$. This linear map operates on the basis vectors $e,f$ as follows
\[T_w(e)=f \quad \text{ and } \quad  T_w(f)=q^3e+(q^3-1)f.\]
From these formulas, we have the following eigenspace decomposition: the eigenvalues of $T_w$ are $-1$ and $q^3$, with the associated spaces given by
\begin{equation}\label{eq:eigenSUw}
W_{w,-1}=\mathbb C(q^3e-f) \quad \text{ and } \quad W_{w,q^3}=\mathbb C(e+f),    
\end{equation}
respectively.

Now we are going to focus on $T_s$. In this case, we have
\[T_s(e)=(q-1)e+(q/a)f\quad \text{ and } \quad T_s(f)=ae,\]where $a=\chi\cdot \delta^{1/2}(\operatorname{diag}(\varpi^{-1},1,\varpi))$. Indeed, let us first compute $T_s(f)$.  By definition we have that
\[T_s(f)=\sum_{u(I\cap sIs^{-1})\in I/(I\cap sIs^{-1})}us\cdot f.\] Thus, we need to identify $I \cap sI s^{-1}$. Conjugation by $s$ gives the following
\[\begin{pmatrix} 0^{\phantom{-1}}&\phantom{-}0&\varpi_F^{-1} \\
0^{\phantom{-1}}&-1&0 \\
\varpi_F&\phantom{-}0&0 
\end{pmatrix}\begin{pmatrix} x_{11}& x_{12}&x_{13} \\    x_{21}& x_{22}&x_{23}  \\
  x_{31}& x_{32} &x_{33}
    \end{pmatrix} \begin{pmatrix} 0^{\phantom{-1}}&\phantom{-}0&\varpi_F^{-1} \\
0^{\phantom{-1}}&-1&0 \\
\varpi_F&\phantom{-}0&0 
\end{pmatrix}=\begin{pmatrix} \phantom{\varpi_F}x_{11}& -\varpi_F^{-1}x_{21}&\phantom{-}\varpi_F^{-2}x_{31} \\
    -\varpi_F x_{12}&\phantom{\varpi^{-2}} x_{22}&-\varpi_F^{-1}x_{32}  \\
    \varpi_F^2x_{13}& -\varpi_F^{\phantom{-1}}x_{23} &\phantom{-\varpi_F^{-1}}x_{33}
    \end{pmatrix}.\]
 So, if $N$ is the unipotent subgroup given by the bottom left coefficient, the inclusion of $N(\mathfrak{p}_F)$ into $I$ induces a bijection of $N(\mathfrak p_F)/N(\mathfrak p_F^2)$ onto $I/I \cap sI s^{-1}$. It follows that 
\[ T_s(f)(x)= \sum_{vN(\mathfrak p_F^2)\in N(\mathfrak p_F)/N(\mathfrak p_F^2)}f(xvs).\]
Since $wvs\in B$ for every $v\in N(\mathfrak p_F)$, $T_sf(w)=0$ and the only summand that contribute in $T_s(f)(1)$ is from the contribution of $N(\mathfrak p_F^2)$. Thus, we find $T_s(f)(1)=a=\chi\cdot \delta(\operatorname{diag}(\varpi^{-1}_F,1,\varpi_F))$, and $T_s(f)=ae$. Consequently, $T_s(e)=(q-1)e+(q/a)f$. 

From the formula we just checked, we have the following eigenspace decomposition: the eigenvalues of $T_s$ are $-1$ and $q^3$, with associated space given by
    \begin{equation}\label{eq:eigenSUs}
W_{s,-1}=\mathbb C(ae-qf) \quad \text{ and } \quad W_{s,q^3}=\mathbb C(ae+f),
\end{equation}
respectively.

Thanks to the equivalence between Iwahori blocks  and the categories of Iwahori Hecke modules \cite{B1976}, there is reducibility if and only if $T_w$ and $T_s$ have a common eigenspace. Thus, if we combine equations \eqref{eq:eigenSUw} and \eqref{eq:eigenSUs}, we check that this happens in four cases:

\begin{enumerate}
    \item $a=1$ with the trivial representation as a subrepresentation. Indeed, in this case the subspace $\mathbb C(e+f)$ invariant under the action of the Iwahori Hecke algebra. As the fixed by $I$ vector of the trivial representation is indeed $\mathbb C(e+f)$ and the action of the Iwahori Hecke algebra is given by multiplication by $q^3$, we have a non-zero morphism  between representations of $G$, from the trivial representation to $i_B^G(\chi)$.
    \item $a=q^4$ with the trivial representation as a quotient. Indeed, we note that the corresponding $a$ of the the dual representation is equal to $\chi^{-1}\cdot \rho(\operatorname{diag}(\varpi^{-1}_F,1,\varpi_F))=q^{-2}\cdot q^2=1$. Thus, from point (1), the dual representation has the trivial representation as subrepresentaion. Therefore, we obtain that $i_B^G(\chi)$ has the trivial representation as a quotient.
    \item $a=-q$ with a non-zero $K$-fixed vector in the subrepresentation.
    \item $a=-q^3$ with a non-zero $K'$-fixed vector in the subrepresentation. In this case this representation is of discrete series. Indeed, as the restriction to $T$ of this representation is $\chi\cdot\rho$ and $-q^3=\chi\cdot \rho(\operatorname{diag}(\varpi^{-1}_F,1,\varpi_F))=q^2\chi(\operatorname{diag}(\varpi_F,1,\varpi^{-1}_F))^{-1}$, we have that its central exponent is negative. 
\end{enumerate}

\begin{remark}
We note that part (4) of the previous discussion is an example of a phenomenon explained in Remark \ref{rmk:specialdiscrete}.
\end{remark}

\medskip

\noindent{\sc \Small H\'ector de Castillo, Departamento de Matem\'atica y Ciencia de la Computaci\'on, Universidad de Santiago de Chile, Las Sophoras 173, Estaci\'on Central, Santiago}\\
\emph{\Small E-mail address: }\texttt{\Small hector.delcastillo@usach.cl}

\noindent{\sc \Small Guy Henniart, Laboratoire de
Mathématiques d’Orsay, Université-Paris-Saclay \& CNRS, Bâtiment 307, Orsay cedex F-91405, France}\\
\emph{\Small E-mail address: }\texttt{\Small Guy.Henniart@universite-paris-saclay.fr}

\noindent{\sc \Small Luis Alberto Lomel\'i, Instituto de Matem\'aticas, Pontificia Universidad Cat\'olica de Valpara\'iso, Blanco Viel 596, Cerro Bar\'on, Valpara\'iso, Chile}\\
\emph{\Small E-mail address: }\texttt{\Small Luis.Lomeli@pucv.cl}

\end{document}